\documentclass[english]{article}
\usepackage{custom_tex}
\usepackage{todonotes}

\title{Consistency of invariance-based randomization tests}

\date{\today}
\author{Edgar Dobriban\footnote{Department of Statistics and Data Science, The Wharton School, University of Pennsylvania. \texttt{dobriban@wharton.upenn.edu}.}}
\graphicspath{
{exp/}
{exp/vector_sign/}
{exp/vector_rotation/}
{exp/t-test/}
}

\begin{document}

\maketitle

\abstract{Invariance-based randomization tests---such as permutation tests, rotation tests, or sign changes---are an important and widely used class of statistical methods. They allow drawing inferences under weak assumptions on the data distribution. Most work focuses on their type I error control properties, while their consistency properties are much less understood.

We develop a general framework and a set of results on the consistency of invariance-based randomization tests in signal-plus-noise models. Our framework is grounded in the deep mathematical area of representation theory. We allow the transforms to be general compact topological groups, such as rotation groups, acting by general linear group representations. We study test statistics with a generalized sub-additivity property.
We apply our framework to a number of fundamental and highly important problems in statistics, including sparse vector detection,  testing for low-rank matrices in noise, sparse detection in linear regression, and two-sample testing. Comparing with minimax lower bounds, we find perhaps surprisingly that in some cases, randomization tests detect signals at the minimax optimal rate.
}

\tableofcontents
\medskip

\section{Introduction}
Invariance-based  randomization tests---such as permutation tests---are an important, fundamental, and widely used class of statistical methods. They allow making inferences in general settings, with few assumptions on the data distribution. Most methodological and theoretical work focuses on their validity, studying their  type I error (false positive rate) control. There is also work on their robustness properties, but less is known about their power and consistency properties. 

Our work develops a general theoretical framework to understand the consistency properties of invariance-based randomization tests in signal-plus-noise models. In particular, we allow the randomization distributions to be Haar measures over general compact topological groups, such as rotation groups. We go beyond most prior work, which focuses on discrete groups (mainly permutation groups), and does not fully develop the technically challenging case of compact groups. Moreover, we allow the action of these groups on the data to be via arbitrary compact linear group representations.

We apply our theoretical framework to a number of fundamental and highly important problems in statistics, including sparse vector detection,  low-rank matrix detection, sparse detection in linear regression,  and two-sample testing. 
Perhaps surprisingly, we find that invariance-based randomization tests are minimax rate optimal in a number of cases. 
The reason why we consider this to be surprising is that the randomization tests are constructed using the same universal principle. They have only minimal information about the problem, namely a set of symmetries of the noise, and a test statistic that is expected to be ``large" under the alternative. 

In more detail, our contributions are as follows:
\benum
\item {\bf Representation-theoretic framework.} We develop a framework for consistency of invariance-based randomization tests based on group representation theory. In our framework, we have a compact topological group that acts linearly on the data space. We assume that under the null hypothesis,  distribution of the data is invariant under the action of the group. We sample several group elements chosen at random from the Haar measure on the group, and apply them to the data. We consider the standard invariance-based randomization test which rejects the null hypothesis when a chosen test statistic is larger than an appropriate quantile of the values of the test statistic applied to the randomly transformed data.

\item {\bf Consistency results.}  We develop consistency results for the invariance-based randomization test in signal-plus-noise models. We consider sequences of signal-plus-noise models where the signal equals zero under the null hypothesis. We study broad classes of test statistics {satisfying the weak requirement of so-called $\psi$-subadditivity. This includes, 
for instance, suprema of linear functionals of the data, 
norms and semi-norms, concave non-decreasing functions in one dimension, convex functions of bounded growth.
Further, this class is closed under conic combinations, taking maxima, and compositions with one-dimensional nondecreasing sub-additive functions.}

We develop a general consistency result, showing that if the sequence of alternatives is such that the value of the test statistic is large enough, then the test rejects with probability tending to unity.
{We compare this to the corresponding result for the deterministic test based on the same the statistic.}
The consistency threshold is inflated slighly by a signal-noise interference effect. By randomly transforming the signal, we create additional noise, inflating the effective noise level in the randomized statistic compared to its distribution under the null. However, we later show that in many examples this inflated noise level can be controlled.
 As part of our consistency theory,  we extend to the setting with nuisance parameters, which allows us to handle problems such as two-sample testing. 

\item {\bf New proof techniques.}  Our proofs are based on novel approaches. For the proofs of the general consistency result, we proceed by a series of reductions, first reducing from the quantile of the randomization distribution to its maximum, then from considering several random transformations to only one transform, and then reducing from a dependent transformed signal and noise to independent ones. 

\item {\bf Examples.} We illustrate our results in several important examples. We show that our results provide consistency conditions for invariance-based  randomization tests in a number of problems, including sparse vector detection,  low-rank matrix detection, sparse detection in linear regression, and two-sample testing.

{For sparse vector detection we consider two settings: where the noise vectors for the different observations are independent and sign symmetric (but not necessarily identically distributed), and where they are independent and rotationally symmetric (spherical). For both cases we obtain general consistency results, and some matching lower bounds. Specifically for the sign symmetric case where the entries of the noise are independent and identically distributed according to a sub-exponential distribution, our upper bound for the signflip randomization test matches a lower bound that we obtain. For spherical noise we obtain general upper bounds as well as specific examples for multivariate $t$ distributions. We also provide similar results for two sample testing.}

{For low rank matrix detection, we consider the case where each of the columns of the noise matrix as an independent spherical distribution. We obtain a general upper bound for the operator norm test statistic, using the associated rotation test. We show that this result is rate optimal for the special case of normal noise.}

{For sparse vector detection in linear regression, we study detection based on the $\ell_\infty$ norm of the least-squares estimator. We assume that the noise entries of each observation are independent and sign-symmetric. We provide a consistency result for the associated signflip based randomization test, in terms of geometric quantities determined by the feature matrix; namely the suprema of two associated Bernoulli processes.} 
\eenum
As a general conclusion, we think it is perhaps surprising that invariance-based randomization tests can sometimes adaptively detect signals at the same rate as the optimal tests that assume knowledge about the exact noise distribution. 
We support our claims with numerical experiments.

{\bf Note on terminology.} We follow the terminology of ``randomization tests" from Ch. 15.2 of the standard textbook by \cite{lehmann2005testing}: ``the term randomization test will refer to tests obtained by recomputing a test statistic over transformations (not necessarily permutations) of
the data." This does not consider tests based on randomization of treatments; see e.g., \cite{onghena2018randomization,hemerik2020another} for discussion. In particular, \cite{hemerik2020another} suggest using ``randomization tests" only when the treatments are randomized, and suggest using ``group invariance tests" for the type of tests we consider. 
For consistency with the standard textbook by \cite{lehmann2005testing}, we will simply use the terminology ``invariance-based randomization tests" or ``randomization tests". {Another well-known example of randomization occurs with discretely distributed tests, to ensure exact type I error control; our work is unrelated to this issue.}

{{\bf Some notations.} For a positive integer $m\ge 1$, the $m$-dimensional all-ones vector is denoted as $1_m = (1,1,\ldots,1)^\top$. We denote $[m]:=\{1,2,\ldots,m\}$, and for $j\in [m]$, the $j$-th standard basis vector by $e_j = (0,\ldots, 1, \ldots, 0)$, where only the $j$-th entry equals unity, and all other entries equal zero. The variance of a random variable $X$ is denoted as $\V{X}$ or $\Var{X}$. For two random vectors $X,Y$, we denote by $X=_dY$ that they have the same distribution.
For an index $m=1,2,\ldots$, and two sequences $(a_m)_{m\ge 1}, (b_m)_{m\ge 1}$, $a_m \lesssim b_m$ (and $a_m = O(b_m)$) means that $a_m \le  Cb_m$ for some $C \ge 1$ independent of $m$, but possibly dependent on other problem parameters as specified case by case. 
We write $a_m \gtrsim b_m$ (or $a_m = \Omega(b_m)$) when  $b_m \lesssim a_m$, and $a_m \sim b_m$ (or $a_m = \Theta(b_m)$) when $a_m \lesssim b_m \lesssim a_m$.
For a vector $v \in \R^m$, and $p\in(0,\infty)$, $\|v\|_p$ denotes the $\ell_p$ norm. Unless otherwise specified, $\|v\|$ denotes the Euclidean or $\ell_2$ norm, $\|v\|=\|v\|_2$. 
For a matrix $s$, the norm $\|s\|_{2,\infty}$ is the maximum of the column $\ell_2$ norms of $s$.
For two subsets $A,B$ of a vector space, $A+B = \{a+b:a\in A,b\in B\}$ denotes the Minkowski sum. 
For a $p\times 1$ vector $v$, let $M=\diag(v)$ be the $p\times p$ diagonal matrix with the entries $M_{ii}=v_i$.
A function $f:V\mapsto V'$, where $V,V'$ are two vector spaces, is an odd function if $f(-v) = - f(v)$ for all $v\in V$.
A Rademacher random variable is uniform over the set $\{\pm 1\}$. 
For a probability distribution $Q$ and a random variable $X\sim Q$, we may write probability statements involving $X$ in several equivalent ways, for instance for the probability that $X$ belongs to a measurable set $A$, we may write: $P(X\in A)$, $P_X(A)$, $P_{Q}(X\in A)$, $P_{X\sim Q}(X\in A)$, $Q(X\in A)$, or $Q(A)$. 
Further, if $Q$ belongs to a collection of probability measures $H$ (e.g., a null or an alternative  hypothesis), then we may also write $P_{H}(A)$ to denote $Q(A)$ for an arbitrary $Q\in H$.
}

\subsection{Related works}
There is a large body of important related work. Here we can only review the most closely related ones due to space limitations. The idea of constructing a statistical test based on randomly chosen permutations of iid samples in a dataset dates back at least to \cite{eden1933validity,fisher1935design}, see \cite{david2008beginnings,berry2014chronicle} for historical details. General references on permutation tests include \cite{pesarin2001multivariate,ernst2004permutation,pesarin2010permutation,pesarin2012review,good2006permutation,anderson2001permutation,kennedy1995randomization,hemerik2018exact}.
These tests have many applications, for instance in genomics \citep{tusher2001significance} and neuroscience \citep{winkler2014permutation}.  For more general discussions of invariance in statistics see \cite{eaton1989group,wijsman1990invariant,giri1996group}; for a general probabilistic reference see also \cite{kallenberg2006probabilistic}.

Two-sample permutation tests date back at least to \cite{pitman1937significance}, and have recently been studied in more general multivariate contexts \citep{kim2020robust}. This problem brings special considerations such as issues with using balanced permutations \citep{southworth2009properties}.

A number of invariance-based randomization based tests have been developed for linear and generalized linear models \citep{freedman1983nonstochastic,perry2010rotation,winkler2014permutation,hemerik2020robust}. The works by \cite{anderson1999empirical,winkler2014permutation} review and compare a number of previously proposed permutation methods for inference in linear models with nuisance parameters.  \cite{hemerik2020permutation} show empirically that permutation tests can control type I error even in certain high dimensional linear models.  \cite{hemerik2020robust} develop tests for potentially mis-specified generalized linear models by randomly  flipping signs of score contributions. 

Other specific problems where invariance-based randomization tests have been developed include independence tests \citep{pitman1937significance2}, location and scale problems \citep{pitman1939tests}, parallel analysis type methods for PCA and factor analysis \citep{horn1965ara,buja1992,dobriban2017permutation,dobriban2019deterministic}, and time series data, where \cite{jentsch2015testing} randomly permute entries between periodograms to test for equality of spectral densities. In addition, randomization based inference has been useful to study factorial designs \citep{pauly2015asymptotic}, regression kink designs \citep{ganong2018permutation}, and linear mixed-effects models \citep{rao2019permutation}.

For the theoretical aspects of invariance-based randomization tests, \cite{lehmann1949theory} develop results for testing a null of equality in distribution $H_{m0}:X_m =_d g_mX_m$ where a transform $g_m\in \mathcal{G}_m$ is chosen from a group $\mathcal{G}_m$ acting on the data. 
They show that all admissible tests have constant rejection probability equal to the level over each orbit, i.e., are similar tests. 
They use this to show that the most powerful tests against simple alternatives with density $f_m$ reject when $f_m(X_m)$ is greater than the appropriate quantile of $\{f_m(g_mX_m),g_m\in \mathcal{G}_m\}$. They use the Hunt-Stein theorem to derive uniformly most powerful (or most stringent) invariant tests from maximin tests for testing against certain composite alternatives. These are related to our results, but we focus on consistency against special structured signal-plus-noise alternatives instead of maximizing power in finite samples.  

The seminal work by \cite{hoeffding1952large} considers general group transforms, including signflips, for testing symmetry of distributions, but focuses on permutation groups for most part. The main results center on power and consistency of tests. For consistency, the main result (Theorem 2.1 in \cite{hoeffding1952large}) states that for a test statistic $f_m$ such that $f_m(x)\ge 0$ and $\E_{G_m\sim   Q_m} f_m(G_mx) \le c$ where $G_m\sim Q_m$ denotes that the group element is distributed according to the probability measure $Q_m$ on $\mathcal{G}_m$, we have $q_{1-\alpha,m}(x) \le c/\alpha$, $\alpha\in (0,1)$, where $q_{1-\alpha,m}$ is the $1-\alpha$-th quantile of the distribution of $f_{m}(G_mx)$ when $G_m\sim   Q_m$. Then, if $f_m\to\infty$ under a sequence of alternatives, the test that rejects when $f_m>q_{1-\alpha,m}(x)$ is consistent, i.e. has power tending to unity. 
{These conditions are distinct from ours. Specifically, his conditions require the test statistic to be pointwise bounded (for each datapoint $x$, they require that $\E_{G_m\sim   Q_m} f_m(G_mx) \le c$), whereas our conditions are high probability bounds in terms of the randomness over both the data and the random transform. 
Thus, the two types of conditions are different. Our condition does not even require that the expectation $\E_{G_m\sim   Q_m} f_m(G_mx)$ is finite, and thus are applicable to heavy-tailed distributions.}

For asymptotic power of certain special invariance-based randomization tests, one can obtain results based on contiguity, see e.g., Example 15.2.4 in \cite{lehmann2005testing}. However, we are interested in problems where the contiguity of the alternatives  may be unknown, or hard to establish.

For permutation tests, \cite{dwass1957modified} shows that it is valid to randomly sample permutations---as opposed to using all permutations---to construct the randomization test. \cite{hemerik2018false} provide a general type I error control result for random group transformations under exact invariance, and apply it to false discovery proportion control.  See also \cite{hemerik2019permutation}.  \cite{hemerik2018exact} extend this in various forms, including to sampling transforms without replacement, and giving rigorously justified formulas for $p$-values.

Most works assume exact invariance of the distribution. \cite{romano1990behavior} studies the behavior of invariance-based randomization tests beyond the exact group invariance framework. This work shows that asymptotic validity holds in certain cases, and fails in others. \cite{canay2017randomization} relax assumptions to only require invariance in distribution of the limiting distribution of the test statistic.
They show that the group randomization test has an asymptotically correct level.

\cite{chung2013exact} develop general permutation tests with finite-sample error control based on studentization. Further studies include discussions of 
conditioning on sufficient statistics \citep{welch1990construction}, combination methods \citep{pesarin1990nonparametric}, and others \citep{janssen2003bootstrap,kim2020minimax}.

Beyond permutation tests, 
flipping signs is considered in many works, see e.g., \cite{pesarin2010permutation}.
Following \cite{wedderburn},  \cite{langsrud2005rotation} discusses rotation tests in Gaussian linear regression.  This approach assumes data $X_m\sim \N(0,I_{p_m} \otimes \Sigma_m)$, and computes the values of test statistics on $X_R = R_mX_m$, where $R_m$ are uniformly distributed orthogonal matrices over the symmetric group $O(p_m)$.  This is applied to testing independence of two random vectors, as well as to more general tests in multivariate linear regression. \cite{perry2010rotation} extends the method to verify latent structure. \cite{solari2014rotation} argues for the importance of this method in  multiple testing adjusting for confounding. The theoretical aspects of rotation tests for sphericity testing of densities are discussed briefly by \cite{romano1989bootstrap}, Proposition 3.2.

\cite{toulis2019life} develops residual invariance-based randomization methods for inference in regression. 
This work considers a general invariance assumption $\ep_m=_d g_m\ep_m$ for the noise $\ep_m$, for all group elements $g_m\in \mathcal{G}_m$. 
For ordinary least squares (OLS), it considers the test statistic $t(\hat\ep_m)=a^\top (X_m^\top X_m)^{-1} X_m^\top g_m\hat\ep_m$, where $\hat\ep_m$ are the OLS residuals, and $a$ is a vector. This work discusses many examples, including clustered observations such that the noise is correlated within clusters, proposing to flip the signs of the
cluster residuals. 

There are a number of works studying the power properties of invariance-based randomization tests. We have already discussed the fundamental work by \cite{hoeffding1952large}. \cite{pesarin2010finite} develop finite-sample consistency results for certain  combination-based permutation tests for multivariate data, when the sample size is fixed and the dimension tends to infinity. They focus on one-sided two-sample tests, and discuss Hotelling's $T$-test as an example.  \cite{pesarin2013weak} characterize weak consistency of permutation tests for one-dimensional two-sample problems. They study stochastic dominance alternatives assuming the population mean is finite and without assuming existence of population variance.

 \cite{pesarin2015some} develops some further theoretical aspects of permutation tests. This includes consistency properties (Property 9), for two-sample tests under some non-parametric assumptions, and alternatives specified by an increased mean of the test statistic. These have different assumptions than the results in our paper, focusing on two-sample problems (while we have general invariance), and non-parametric models (while we focus on parametric ones). 

{One one of the most closely related papers is that of \cite{kim2020minimax}. We discuss the similarities and differences. The methodology of permutation tests is a special case of general group invariance tests; however the examples in our work mostly concern sighflip-based and rotation tests. The only overlap in the specific problems studied is for two-sample testing, but under different assumptions (we study testing the equality of means in a location model, whereas they study testing the equality of two distributions such as multinomials and distributions with Holder densities). Thus our results are not directly comparable. For instance our minimax optimality for two-sample testing involves location families with IID sub-exponential noise, whereas their examples are multinomial distributions and Holder densities.}


{
{\bf In context.}  To put our work in context, we can make the following comparisons:
\bitem
\item The vast majority of works only provide theoretical results for the behavior of randomization tests under the null hypothesis \citep{eden1933validity,fisher1935design,pitman1939tests,hemerik2018exact}.
\item The seminal work of \cite{hoeffding1952large}, already discussed above, provides consistency conditions that are distinct from ours; and does not discuss any of the specific problems that we study. The power analysis based on contiguity \citep{lehmann2005testing} does not apply to many of the problems we study. As detailed above, the consistency and minimax optimality analyses from \cite{pesarin2010finite,pesarin2013weak,pesarin2015some,kim2020minimax} all concern different setups from and/or special cases of our results.
\eitem
}

{
{\bf Scientific context.}  For an even broader scientific context, we emphasize that randomization tests are ubiquitous in modern science. 	Their proper use is crucial for reproducible results; and failure to use them correctly can result in irreproducible results, false scientific discoveries, and ultimately a waste of resources. Here are some examples:
\bitem
\item  In neuroscience, the analysis of fMRI data requires testing hypotheses about the activation of regions in the brain. It has been observed that inferences based on models such as Gaussian fields with parametric covariance functions can have massively inflated false-positive rates \citep{eklund2016cluster}. To mitigate this problem, it has been proposed to use randomization methods such as permutation methods (for two-sample problems) or random sign flips (for one sample problems) to set critical values. Further randomization methods  have been proposed for other problems such as general linear models \citep{winkler2014permutation}, or brain network comparison \citep{simpson2013permutation}. \\
The ultimate goal is to report reliable discoveries, which involves analyzing data not from the null distribution, but rather from an alternative distribution that contains signals. Our work can shed light on when using randomization tests in such an analysis from data containing signals can succeed.
\item In genetics and genomics, hypothesis testing is routinely performed to identify associations between observed phenotypes and genotypes, or between genotypes, etc. Randomization tests, and in particular permutation tests, are widely used to set critical values, in methods such as transmission disequilibrium tests, 
 etc, and are broadly available in popular software such as PLINK, see for instance \cite{churchill1994empirical,purcell2007plink,epstein2012permutation}. Randomization tests are also used for more sophisticated tasks such as gene set enrichment analysis \citep{subramanian2005gene,barry2005significance,efron2007testing}.
\eitem
}

\section{General framework}
\label{nuis}

\subsection{Setup}

We consider a sequence of statistical models, indexed by an index parameter $m \to\infty$. 
We observe data $X_m$ from a real vector space $V_m$, for instance a vector or a matrix belonging to Euclidean space $\R^{p_m}$. We assume that we know a group $\mathcal{G}_m$ of the symmetries of the distribution of the data. 
{See Section \ref{prac} for a discussion of how is can arise in practice.}
A group $\mathcal{G}_m$ has a multiplication operation $"\cdot"$ that satisfies the axioms of associativity, identity, and invertibility. For instance, we could have that the entries of $X_m$ are exchangeable (corresponding to the permutation group), symmetric about zero (corresponding to the group of addition modulo two) or that the density of $X_m$ is spherical (corresponding to the rotation group). 

In addition, to transform the data, we have a group representation $\rho_m: \mathcal{G}_m \to GL_m(V_m)$, acting linearly on $X_m\in V_m$ via $g_mX_m:=\rho_m(g_m)\cdot X_m$. The group representation ``represents" the elements of the group $\mathcal{G}_m$ as invertible linear operators $V_m\mapsto V_m$ belonging to the general linear group $GL_m(V_m)$ of such operators. 
The group representation $\rho_m$ preserves the group multiplication operation, i.e., $\rho_m(g_mg_m') = \rho_m(g_m)\rho_m(g_m')$ for all $g_m,g_m'\in \mathcal{G}_m$, and $\rho_m(e_{\mathcal{G}_m})=I_{V_m}$, where $e_{\mathcal{G}_m}$ is the identity element of the group, and $I_{V_m}$ is the identity operator on $V_m$. 
For general references on representation theory, see \cite{serre1977linear,james2001representations,fulton2013representation,hall2015lie,knapp2013lie,eaton1989group}, etc. For group representations in statistics, see \cite{diaconis1988group}. We will use basic concepts from this area throughout the paper.

{\bf Null hypothesis of invariance, and randomization test.} We want to use the symmetries of the noise distribution to detect the presence of non-symmetric signals. Under the null hypothesis, we assume that the distribution of the data is invariant under the action of each group element $g_m\in \mathcal{G}_m$: $X_m=_d g_mX_m$.
\footnote{This is called the ``Randomization hypothesis", Definition 15.2.1 in \cite{lehmann2005testing}.} 
We study the following invariance-based randomization test (sometimes also called a group invariance test), which at various levels of generality has been considered dating back to \cite{eden1933validity,fisher1935design,pitman1937significance,lehmann1949theory,hoeffding1952large}. 
We sample $G_{m1},\ldots,G_{mK}$ iid from $\mathcal{G}_m$ (in a way specified below), and reject the null if for a fixed test statistic $f_m:V_m\mapsto \R_m$, the following event holds
\beq\label{t0}
\mathcal{E}_m=\{f_m(X_m)>q_{1-\alpha}\left(f_m(X_m),f_m(G_{m1}X_m),\ldots, f_m(G_{mK}X_m)\right)\},
\eeq
for the $1-\alpha$-th quantile $q_{1-\alpha}$ of the numbers $\{f_m(X_m),f_m(G_{m1}X_m),\ldots, f_m(G_{mK}X_m)\}$ and some $\alpha\in(0,1]$. 
Specifically, let $G_{m0}=I_{V_m}$ be the identity operator on $V_m$, and $f_{(1)} \le f_{(2)} \le \ldots \le  f_{(K+1)}$ be the ordered test statistics of the set $\{f_m(G_{mi}X_m), i\in \{0,1,\ldots,K\}\}$. Let $k = \lceil (1-\alpha) (K+1)\rceil$. 
Rejecting the null if $f_m(X_m)> f_{(k)}$ is guaranteed to have level at most $\alpha$, see e.g., theorem 2 in \cite{hemerik2018exact} for an especially clear and rigorous statement. 
In some cases, one can relax this to assume only $f_m(X_m)=_d f_m(g_mX_m)$ under the null, see e.g., \cite{canay2017randomization,hemerik2018exact}, but we will not pursue this.

{\bf Noise invariance and robustness.} 
The advantage of randomization tests compared to a rejection region of the form $f_m(X_m)>\tilde c_m$ for a fixed $\tilde c_m$ is that it does not require the manual specification of the critical value $\tilde c_m$.
The critical value needs to account for the set of distributions included the null hypothesis, which may be a very large nonparametric family. In this case, it might be challenging to set the critical value to ensure type I error control. Randomization tests avoid this problem by relying on the symmetries of the noise distributions.
To wit, randomization tests are valid under \emph{any} null hypothesis for which the distribution of the noise is invariant under the group. This effectively amounts to that only depend on the collection of \emph{orbits}, which form a maximal invariant of the group, see Sections 3 and 4 in \cite{eaton1989group} for examples.

For instance, for the rotation group $O(p_m)$, we get \emph{spherical} distributions, which have a density $p_m(X_m)=\pi_m(\|X_m\|_2)$ with respect to a $\sigma$-finite dominating measure on $\R^{p_m}$ only depending on the Euclidean norm of the data $X_m$ \citep{kai1990generalized,gupta2012elliptically,fang2018symmetric}. This is a non-parametric class that includes in particular distributions such as the multivariate $t$, multivariate Cauchy, scale mixtures of spherical normals etc. In particular, it includes heavy tailed distributions, for which tests based on the normal assumption can have  inflated type I error. 
As another example, consider a stationary field $X_{m,J} = (X_i)_{i\in J}$, for some index set $J$. Suppose $\mathcal{G}_m$ acts on $J$, and induces an action on $X_{m,J}$ via its regular representation, i.e., $(g_mX_{m,J})_i = X_{g_m^{-1}i}$. For instance, we can have a discrete-time stationary time series where $J = \mathbb{Z}$, and $\mathcal{G}_m = (\mathbb{Z},+)$. In this example, any translation of the time series keeps the distribution invariant; but this allows a wide range of noise distributions. 

While sometimes it is possible to construct test statistics whose distribution does not depend on a broad set of null hypotheses (see e.g., Section 4.3 ``Null robustness" in \cite{eaton1989group}), this may not be possible when the null hypothesis has a great number of nuisance parameters. For example, this holds for null hypotheses where each noise entry is independent with a probability density only assumed to be symmetric around zero, in which case sign-flip based methods are applicable, see e.g., Example 15.2.1 of \cite{lehmann2005testing}, and also \cite{hemerik2020robust,hong2020selecting}.



{\bf Haar measure.}  In the definition of the randomization test, $G_{m1},\ldots,G_{mK}$ are chosen iid from the uniform (Haar) measure on $\mathcal{G}_m$, which is assumed to exist. We refer to Section 2 in \cite{folland2016course} for details, see also \cite{fulton2013representation,eaton1989group,wijsman1990invariant}. Thus, $\mathcal{G}_m$ is assumed to be a compact Hausdorff topological group  with the Borel sigma-algebra generated by the open sets. For brevity, we will sometimes refer to such groups as \emph{compact groups}. 
The Haar probability measure $Q_m$ on $\mathcal{G}_m$ is the unique probability measure such that $Q_m(G_m\in A) = Q_m(G_m\in g_m'A)$ for all $g_m'\in \mathcal{G}_m$ and for all Borel sets $A_m$.  See e.g., Theorems 2.10 \& 2.20 in \cite{folland2016course}. Thus, in particular, we have the equality in distribution $G_m=_dG_mg_m'$ for $G_m\sim Q_m$, and any fixed $g_m'\in \mathcal{G}_m$. 

{{\bf Choice of $K$.} We remark that, as is well known, choosing $K$ larger, and $k$ as above, can generally lead to a more precise control of the type I error. Indeed, for a given $K$, the smallest type I error control guaranteed by the randomization test is $1/(K+1)$, and there are only $K$ possible values of $k\in [K]$ to control the type I error more generally. Thus, for a larger $K$, we expect that we can control the type I error more accurately. Indeed, we observe this in our experiments.}

{\bf Alternative hypothesis: signal-plus-noise model.} To study the consistency of the test, we will consider a sequence of alternative hypotheses in the signal-plus-noise model {with a deterministic signal $s_m$ and a random noise $N_m$}
$$X_m=s_m+N_m.$$
The null hypothesis is specified by $H_{m0}:s_m=0_{p_m}$, in which case $X_m=N_m$. The alternative hypothesis $H_{m1}$ is specified by a set $\Theta_{m1}\subset V_m$ of signals $s_m\in\Theta_{m1}$. We call $\Theta_m = \{0\} \cup \Theta_{m1}$ the parameter space. The alternative hypothesis is decisively \emph{not} invariant under $\mathcal{G}_m$. In fact, one can view the test statistic as detecting deviations from invariance. 

{We view the signal-plus-noise model as quite broad, and we will study a variety of examples as special cases. The breadth of the model arises from two aspects: First, one can choose the signal parameter space $\Theta_m$ to be quite general, for instance a linear subspace, a union of linear subspaces, a convex cone, etc. Second, one can model the family to which the distribution of the noise $N_m$ belongs; and our theory will rely on the symmetries of these distributions. Further, based on finite-dimensional asymptotic statistics, we know that asymptotically any sufficiently regular parametric model is well approximated by a normal observation model, which can be viewed as a signal-plus-noise model like ours if the noise distribution does not depend on the signal. 
}

{However, the scope of this model is limited in a few ways. It assumes a specific ``structural model'' for the data, and it is essentially a submodel of a multi-dimensional location family. For instance, it requires the distribution of the noise to be functionally independent on the unknown paramater $s_m$. In some cases, this may be approximately achieved via appropriate variance-stabilizing transforms. In our analysis, this is currently needed to be able to formulate consistency conditions based on only one global distribution of the noise. If the noise distribution can vary in parameter space, we expect that the behavior of randomization tests could be more complex.}  
{We discuss this and further limitations of our work in Section \ref{disc}.}


\subsection{General consistency}
\label{gencons}
{Our basic idea to establish consistency of randomization tests is to find conditions under which the test statistic under the alternative is much larger than the randomized test statistic, i.e., (informally) $f_m(s_m + N_m) \gg f_m(G_m[s_m+N_m])$. We wish to do this by introducing only broadly applicable assumptions. The first key step is to find a lower bound on $f_m(s_m + N_m)$. To achieve this, we make assumptions of $f_m$.}

{For a given constant $\psi>0$, we consider $\psi$-subadditive test statistics, i.e., functions $f_m:V_m\mapsto\R$ such that for all $a,b\in V_m$,
$$\psi \cdot f(a+b) \le f(a) + f(b).$$
Note that typically $\psi \le 1$.
In the current argument, we will use that $f_m(s_m + N_m) \ge \psi f_m(s_m)- f_m(-N_m)$. This allows us to lower bound the value $f_m(s_m + N_m)$ of the test statistic by a main term $\psi f_m(s_m)$ depending only on the signal, and an error term $- f_m(-N_m)$ depending only on the noise (which we will also control).
We will use a similar argument to upper bound the randomized test statistic $f_m(G_m[s_m+N_m])$.
These conditions are enough to guarantee the consistency of tests of the form $f_m(X_m)> \tilde c_m$ for appropriately chosen ``oracle'' critical values $\tilde c_m$ (which are not practically implementable in general); and we will compare the resulting conditions later in this section.\\
Examples of subadditive functions include:
\begin{enumerate}
\item Given any set $W_m \subset V_m$, the suprema of linear functionals
$$f_m(x) = \sup_{w_m\in W_m} w_m^\top x,$$
assumed to be finite-valued functions, are $1$-subadditive. These are the \emph{sublinear functionals} on $V_m$, see e.g., Sect 5.4, Ch 7, and specifically Exercise 7.103 in \cite{narici2010topological}. In particular, affine functions $f(x) = w^\top x + c$ are  $1$-subadditive for any $w\in V_m$ and any $c\ge 0$.
\item For instance, for any norm $\|\cdot\|$ on $V_m$ (with the dependence on $m$ suppressed), we can take $f_m(x) = \|x\|$ by choosing $W_m = \{w_m: \|w_m\|_*\le 1\}$, the unit ball in the dual norm $\|\cdot\|_*$ of $\|\cdot\|$.
\item When $V_m = \R$ is one-dimensional, for any concave non-decreasing function $c:[0,\infty)\to \R$ such that $c(0)\ge 0$, $f:\R\mapsto\R$ given by $f(x) = c(|x|)$ is $1$-subadditive. Examples include $f(x) =|x|^{q}$ for $q \in (0,1]$. See Section \ref{pf-psi} for the argument.
\item Convex functions of bounded growth: If $f:\R^p \to \R$ is convex and satisfies $\psi f(2x)\le 2 f(x)$, then $f$ is $\psi$-subadditive. Indeed, $f(a) + f(b) \ge 2 f([a+b]/2) \ge \psi f (a+b)$ by convexity and bounded growth. For instance, $f(x) = \|x\|_q^q$, for $q\ge 1$ satisfies $f(2x) = 2^q f(x)$, thus it is $2^{1-q}$-subadditive.
\end{enumerate}}

{
Non-examples include functions of very fast growth, for instance $f: \R\mapsto\R$, $f(x) = \exp(x)$. However, for the purposes of hypothesis testing, only the acceptance and rejection regions are relevant; and thus even for test statistics that are not subadditive, one may---on a case-by-case basis---find sub-additive test statistics with the same acceptance and rejection regions; where our theory can be applied. For instance, instead of the exponential map above, one may consider the identity map.  \\
Further, this class has a number of closure properties, being closed under:
\begin{enumerate}
\item Conic combinations: If $f_j:V_j\mapsto [0,\infty)$, $j\in [J]$ are $\psi_j$-subadditive, then for any $\tau_j\ge 0$, $j\in[J]$, $\sum_{j\in[J]} \tau_j f_j$ is $\min_{j\in [J]} \psi_j$-subadditive.
\item Maxima: If $f_j:V_j\mapsto [0,\infty)$, $j\in [J]$ are $\psi_j$-subadditive, 
then $\max_{j\in[J]} f_j$ is $J^{-1}\min_{j\in [J]} \psi_j$-subadditive.
\item Compositions with 1-D functions: if $f_1:[0,\infty)\mapsto \R$ is non-decreasing and $\psi_1$-subadditive; and  $f_2:\R^p\mapsto [0,\infty)$ is $1$-subadditive, then  $f_1 \circ f_2:\R^p\mapsto \R$ is $\psi_1$-subadditive. Indeed,
$$f_1 \circ f_2(x+y)\le f_1[f_2(x) +f_2(y)]
\le \psi_1^{-1}[f_1\circ f_2(x)+f_1\circ f_2(y)].$$
\end{enumerate}
}

Our first theorem is a general consistency result for randomization tests with $\psi$-subadditive test statistics.

\begin{theorem}[Consistency of randomization test]\label{cons} {Consider a sequence of models indexed by $m\ge1$, $m\in\mathbb{N}$, such that} the data $X_m \in V_m$ follow a $p_m$-dimensional signal-plus-noise model $X_m = s_m+N_m$, where $s_m\in \Theta_m$ {is deterministic} and $N_m$ is a random noise vector. Test the sequence of null hypotheses $H_{m0}:s_m=0$ against a sequence of alternative hypotheses $H_{m1}$ with signal vectors $s_m\in \Theta_{m1}$ for a fixed $\alpha\in(0,1]$.
Reject the null hypothesis using the randomization test \eqref{t0}.
{Let $f_m$ be $\psi$-subadditive.}
Assume the following:
\begin{compactenum}
\item {\bf Noise invariance}. The distribution of the noise is invariant under $\mathcal{G}_m$: $N_m =_d g_mN_m$ for all $g_m\in \mathcal{G}_m$.
\item {\bf Signal strength.} 
There is a sequence $(t_m)_{m\ge1}$, and 
for any sequence $(s_m)_{m\ge 1}$ such that for all $m\ge 1$, $s_m \in \Theta_{m1}$, 
there is another sequence $(\tilde t_m)_{m\ge1}$, 
$\tilde t_m=\tilde t_m(s_m)$, 
such that for all large enough integers $m$, 
\beq\label{cr}
{f_m(s_m)> \psi^{-2} \tilde t_m(s_m) + \psi^{-1}(\psi^{-1}+1)t_m}.
\eeq
Further, as $m \to \infty$,
\begin{compactenum}
\item {\bf Noise level.} {$P(f_m(N_m) \le t_m)\to 1$ and $P(f_m(-N_m) \le t_m)\to 1$.}
\item {\bf Bound on randomized statistic}. {The test statistics evaluated on the randomized signal fall below $\tilde t_m(s_m)$, i.e., for any sequence $(s_m)_{m\ge 1}$ such that for all $m\ge 1$, $s_m \in \Theta_{m1}$,}
$${P_{G_m\sim Q_m}(f_m(G_ms_m)\le \tilde t_m(s_m))\to 1}.$$
\end{compactenum}
\end{compactenum}
Under condition 1, the randomization test has level at most $\alpha$. Under conditions 1\& 2, the randomization test is consistent, {i.e., for the event $\mathcal{E}_m$ from \eqref{t0}, for any sequence $(s_m)_{m\ge 1}$ such that $s_m \in \Theta_{m1}$ for all $m \ge 1$, $\lim_{m\to\infty} P_{G_{m1}, \ldots, G_{mK}\sim Q_m, N_m}(\mathcal{E}_m)=1$}.
\end{theorem}
{
Some comments on the assumptions are in order:
\begin{enumerate}
\item The noise invariance condition is required to ensure the exact type I error control, as discussed above.
\item  Our analysis relies on comparing the size of the test statistic on the data and the randomized data. The sub-additivity assumption allows us to reduce this to comparing the size of the test statistics on the signal, the noise, and the randomized signal. The remaining conditions are meant to capture high-probability deterministic bounds on the statistic over the randomness in the remaining stochastic quantities: the noise and random group elements.
\item The sequence $t_m$ controls the size of the statistic $f_m$ evaluated on the noise $N_m$. The sequence $\tilde t_m(s_m)$ controls the size of the statistic evaluated on the randomized signal $G_m s_m$.
\end{enumerate}
}
See Section \ref{pfcons} for the proof, which is novel. For the consistency result, we proceed by a series of reductions, first reducing from the quantile test to a max-based test, then from considering several random transformations to only one transform, and then reducing from a dependent transformed signal and noise to independent ones.

{\bf Conventions.} To lighten notation, we will often omit the dependence of $\tilde t_m(s_m)$ on $s_m$, writing simply $\tilde t_m$. Further, when it is clear from context what the sequence of tests is, we will simply say that the ``test is consistent", as opposed to saying that the ``sequence of tests is consistent''.

{{\bf Consistency of deterministic test.}  As mentioned, $\psi$-subadditivity is enough to guarantee the consistency of tests of the form $f_m(X_m)> \tilde c_m$ for appropriately chosen critical values $\tilde c_m$. We state this result below and compare it as a  ``baseline" result with the conditions for the consistency of randomization tests.
\begin{proposition}[Consistency of deterministic  test]\label{consdet} In the setting of Theorem \ref{cons}, suppose that condition 2(a) holds, along with the following condition: 
\begin{compactenum}
\item {\bf Signal strength.} There is a sequence
 $(t_m)_{m\ge1}$ such that for all large enough integers $m\ge 1$,
\beq\label{cd}
f_m(s_m)>  2\psi^{-1}t_m.
\eeq
\end{compactenum}
Then, for any sequence
 $(\tilde c_m)_{m\ge1}$ such that $\tilde c_m \le t_m$ for all $m\ge 1$, the sequence of deterministic tests that rejects when $f_m(X_m)> \tilde c_m$ is consistent, i.e., $\lim_{m\to\infty} P_{H_{m1}}(f_m(X_m)> \tilde c_m)=1$.
\end{proposition}
See Section \ref{pfcons2} for the proof. To ensure type I error control at level $\alpha$, the sequence  $(\tilde c_m)_{m\ge1}$ needs to be chosen such that $\sup_{P\in H_{m0}} P(f_m(X_m)> \tilde c_m) \le \alpha$. As we discussed, this can be difficult when the class of null hypotheses is large and has many nuisance parameters. Thus, the deterministic test may not be practically implementable. 
However we can still consider it as an idealized ``baseline", to understand the conditions on the signal strength that our approach provides to ensure consistency. 
Comparing the conditions for data signal strength, \eqref{cr} and \eqref{cd}, and recalling that typically $\psi \le 1$, we see that the requirement for the randomization test is stronger. The factor in front the noise level $t_m$ is larger, and in addition the randomization test also has the additional term $\psi^{-2} \tilde t_m$ controlling the size of the randomized signal. \\
Thus, our requirements for the randomization test are more stringent. However, as explained above, the randomization test requires a method to set the critical value, which may be very hard or impossible in practice in certain problems where the null hypothesis is very large. 
}

{\bf  Nuisance parameters.} We next develop a generalization of our consistency results allowing nuisance parameters. This allows handling problems such as two-sample testing where the global mean is a nuisance.
Let $X_m = \nu_m + s_m + N_m$, where $\nu_m$ is a nuisance parameter, $s_m\in \Theta_m$ is the signal. Suppose $\nu_m$ belongs to a known linear space $U_m$, $\nu_m \in U_m$. 
We can reduce this to the previous setting by projecting into the orthogonal complement of $U_m$. Let $P_m=P_{U_m^\perp}$ be the orthogonal projection operator into the orthogonal complement of $U_m$. Then $P_m\nu_m=0$, so by projecting with $P_m$, we have
$$P_mX_m = P_ms_m + P_mN_m.$$

Let $\tilde{X}_m = P_mX_m$ be the new observation, $\ts_m = P_ms_m$ be the new signal, and $\tN_m =P_mN_m$ be the new noise. 
Then, this reduces to the standard signal-parameter model, with the signal parameter space $\tTheta_m = P_m\Theta_m = \{P_ms_m: s_m\in\Theta_m\}$, and a new induced noise distribution.

\subsection{Review of tools to obtain concrete results}
To analyze concrete examples, we will rely on a few technical tools, reviewed in the following sections.
\subsubsection{Rate optimality}
\label{ro}
{
In this section, we review some basic results on minimax rate optimality for hypothesis testing that we will use, focusing on Ingster's (or the chi-squared) method \citep{ingster1987minimax,ingster2012nonparametric}.
This result allows randomized tests $\phi: V_m \mapsto [0,1]$, where $\phi(x)$ is the probability of rejecting the null for data $x$.
Denote the set of all level $\alpha \in (0, 1)$ tests by
$$\Phi_{m}(\alpha) = \left\{\phi: V_m \mapsto [0,1] : \sup_{P \in H_{m0}} \E_{P}[\phi] \leq \alpha \right\}.$$
Define the minimax type II error as 
$$R_{m} = \inf_{\phi\in\Phi_{m}(\alpha)} \sup_{P \in H_{m1}} \mathbb{E}_{P}[1-\phi].$$
     Suppose that $P_{m0} \in H_{m0}$ and $P_{m1}$, $\dots, P_{mM_m} \in H_{m1}$.
     Define the average likelihood ratio between $P_{m0}$ and $P_{m1}$, $\dots, P_{mM_m}$ as
     $$L_m = \frac{1}{M_m} \sum_{i = 1}^{M_m} \frac{p_{mi}(X_{m})}{p_{m0}(X_{m})},$$
     where $p_{mi}$, $i \in [M_m]\cup\{0\}$ are, respectively, the densities of $P_{mi}$, $i \in [M_m]\cup\{0\}$ with respect to a common dominating sigma-finite measure  on $V_m$.
Then, it is well known (see e.g., \cite{ingster2012nonparametric}, and Section III.B of \cite{banks2018information} for a very clear statement) that to achieve consistency, i.e., to have $R_m\to 0$, we must have $\lim_{m\to\infty}$ $\V_{P_{m0}}[L_m] = \infty$.\\
A further key result holds when the null distribution $P_{m0}$ is $\N(0,I_{p_m})$ and the alternative $H_{m1}$ contains distributions of the form $\N(s_m,I_{p_m})$, for $s_m \in \Theta_{m1}$. Consider a prior $\Pi_m$ on $\Theta_{m1}$. Then, we have, see e.g., \cite{ingster2012nonparametric} or Lemma 1 of \cite{banks2018information}, for two independent copies $S,S'\sim \Pi_m$,
\beq\label{vlr}
\V_{P_{m0}}[L_m]  = \E_{S,S'\sim \Pi_m} \exp(S^\top S').
\eeq
}

\subsubsection{Tail bounds of random variables}
\label{tail}
{
We recall some well known tail bounds for random variables.
Suppose that for all $m \ge 1$ and $i\in [m]$, $Z_{i}$ are iid random variables with a probability distribution $\pi$. Let
$F_\pi(t,n) = P(|n^{-1}\sum_{i=1}^{n} Z_{i}|> t)$, with $Z_i\sim \pi$ iid for all $i\in [n]$. \\
   There is a vast number of well-known results on tail bounds of sums of iid random variables under a variety of conditions, see e.g., \cite{petrov2012sums,boucheron2013concentration,vershynin2018hdp}, etc. Each of these can be used together with our framework to obtain consistency results. In a very rough order of increasing generality:}
 {
   \benum
    \item The tail of sums of \emph{sub-exponential random variables} (including  \emph{sub-gaussian and bounded variables}) can be controlled via Bernstein-type inequalities, which lead to $F_\pi(t,n) \le $ $C\exp(-cn\min\{t,t^2\})$ for some $C,c$ depending only on $\pi$ \citep{vershynin2018hdp}. Bernstein-Orlicz random variables interpolate between sub-Gaussian and sub-exponential random variables \citep{van2013bernstein}.
   \item There are various \emph{Orlicz norms} for random variables, and corresponding tail bounds, for instance for random variables with tail decay of order roughly $\exp(-x^\alpha)$, $\alpha >0$ (which have all polynomial moments but for $\alpha<1$ have no moment generating function) \citep{chamakh2020orlicz}, or of order roughly $\exp(-\ln[x+1]^\kappa)$ for $\kappa>0$ (which have all polynomial moments but no moment $\E_{X\sim \pi} \exp(|X|^c)$, $c>0$ \citep{chamakh2021orlicz}.\\
   For instance, the results of \cite{chamakh2021orlicz} imply the following. Consider $\Psi:\R^+ \mapsto \R^+$, $\Psi(x) = \exp(\ln[x+1]^\kappa)-1$, and for a random vector $Z$, the $\Psi$-Orlicz ``norm''\footnote{This may nor may not satisfy the triangle inequality, see \cite{chamakh2021orlicz} for discussion.}
   $\|Z\|_\Psi = \inf\{c>0: \E \Psi(\|Z\|/c)\le 1\}$. Then, for iid random variables $Z_1,\ldots,Z_n\sim \pi$, with finite $\Psi$-Orlicz norm and finite variance, $F_\pi(t,n) \le 2 \exp(-\ln[Cn_m^{1/2}t+1]^\kappa)$ for some $C$ depending on $\pi$, see the remark after Corollary 2.3 of \cite{chamakh2021orlicz}. 
   \item For random variables with \emph{finitely many polynomial moments}, one has Khintchine-type inequalities \citep{petrov2012sums,boucheron2013concentration}, as well as Rosenthal- and Fuk-Nagaev-type inequalities \citep{rio2017constants,marchina2019rate}.
   \item For more heavy-tailed random variables with only a variance, Chebyshev's inequality applies to the sample mean, but there are tighter tail bounds for other mean estimators, see e.g., \cite{catoni2012challenging,lugosi2021robust,lugosi2019mean}.
   \eenum  
}

\subsubsection{Bernoulli processes}
{
Here we review the definition of Bernoulli processes, which we will use later in our consistency results. For any positive integer $q$, a subset $T$ of $\R^q$, and a vector $b=(b_1,\ldots, b_p)$ of independent Rademacher random variables, $t\mapsto t^\top b$ is referred to as a Bernoulli process (also called a Rademacher process, especially in learning theory) see e.g., \cite{boucheron2013concentration,talagrand2014upper}.\\ 
 In this case, for any function class $F_m = \{f_m^* = (f_{m,1}, \ldots, f_{m,n_m})\}$, such that each $f_{m,i}:\R^{p_m} \mapsto \R$ is an odd function, 
and any random vectors $N_m=(N_{m,1}, \ldots, N_{m,n_m})$ that are mutually independent and sign-symmetric, i.e., $N_{m,j}=_d-N_{m,j}$ for all $j\in [n_m]$,
for iid signflips $b_1, \ldots, b_{n_m}$, conditonal on $N_{m,i} \in \{\pm N_{m,i}^0\}$ for fixed $N_{m,i}^0$, $i\in [n_m]$,
the randomization distribution $(b_1N_{m,1}, \ldots, b_{n_m}N_{m,n_m})$
for test statistics of the form 
 $$f_m(N_m)=\sup_{f_m^*\in F_m} \sum_{i=1}^{n_m} f_{m,i}(N_{m,i})$$
is a Bernoulli process. Indeed, one can take $q = n_m$, and the index set $T = \{(f_{m,1}(N_{m,1}^0)$, $\ldots$, $f_{m,n_m}(N_{m,n_m}^0)): f_m^*\in F_m\}$.\\
 The fundamental result for bounding expectations of suprema of Bernoulli processes is the Bednorz-Latala theorem \citep{bednorz2013suprema}, see also Proposition 5.14 \& Theorem 5.1.5 in \cite{talagrand2014upper} for an expository presentation. Consider a subset $T$ of $\R^q$ for some $q>0$ and a vector $b=(b_1,\ldots, b_q)$ of iid Rademacher random variables. Then, for $Z\sim \N(0,I_q)$, the Bernoulli complexity of $T$ is characterized as
 \beqs
 b(T):=\E\sup_{t\in T} t^\top b \sim \inf\left\{\E\sup_{ t\in T_1} t^\top Z + \sup_{t\in T_2}\|t\|_1:\, T\subset T_1+T_2\right\}.
 \eeqs
 In turn, the Gaussian complexity $\E\sup_{ t\in T_1} t^\top Z$ is characterized up to constants by the generic chaining \citep{talagrand2014upper}.}

{
 Further, Bernoulli processes concentrate around their mean with a sub-Gaussian tail: assuming $T\subset B(t_0,\sigma)$ (where $B(x,r)$ is the $\ell_2$ ball of radius $r$ centered at $x$), for any $u>0$, 
 \beqs
 P(|\sup_{t\in T} t^\top b-b(T)| \ge u) \le c\exp(-cu^2/\sigma^2),
 \eeqs
  for a universal constant $c$, see Theorem 5.3.2 in \cite{talagrand2014upper}. We define the infinum of the radii of all $\ell_2$ balls containing the set $T$ as the radius $r(T)$ of $T$. Further, for any scalar $l$, we denote
  \beq\label{u+}
  U^+(T,l):=b(T) + l \cdot r(T).
  \eeq
  The above results imply that, for any sequence of positive integers $(q_m)_{m\ge1}$, any sequence of sets $(T_m)_{m\ge 1}$ with $T_m \subset \R^{q_m}$, and any sequence $(l_m)_{m\ge1}$ such that $l_m>0$ for all $m$ and $l_m \to \infty$ as $m\to \infty$,
$P(\sup_{t\in T_m} t^\top b \le U^+(T_m,l_m))\to 1.$
  In principle, these results provide basic tools to control the tails of Bernoulli processes. However they can require some work to use in specific cases; thus more specific results (which we will discuss later) are of interest.	
}

\section{Examples}
In this section we apply our theory to several important statistical problems. Our results allow us to determine consistency conditions in a broad range of settings.
\subsection{Detecting sparse vectors}
\label{se1}
Our first example is the fundamental statistical problem of sparse vector detection. We make $n_m$ noisy observations $X_{m,i}$, $i=1,\ldots,n_m$ of a signal vector $s_m$. We assume that the signal vector is either zero, or ``sparse" in the sense that it has only a few nonzero coordinates. We are interested to detect---or test---if there is indeed a nonzero signal buried in the noisy observations. 
This is challenging due to the potentially large and unknown level of noise. Randomization tests can be useful, because they do not require the user to know the level of noise. Indeed, they only require one to know some symmetries of the noise, and automatically adapt to the other nuisance parameters such as the noise level.

{Formally, we observe $n_m$ vectors $X_{m,i}=s_{m}+N_{m,i}$, $i=1,\ldots,n_m$ of dimension $p_m$, which are sampled from a signal-plus-noise model. 
We arrange them into an $n_m \times p_m$ matrix $X_m$, which has the form $X_m =1_{n_m}s_m^\top+N_m$.  
We are interested to detect ``sparse" vectors $s_m$; more specifically, we are interested to test against $s_m$ with a large $\ell_\infty$ norm $ \|s_m\|_\infty$. 
We use the test statistic $f_m(X_m) = n_m^{-1}\|1_{n_m}^\top X_m\|_\infty$.} 

\subsubsection{Sign-symmetric noise}
\label{sign-noise}

{Based on specific assumptions on the noise, various different randomization tests are valid. To illustrate our theory, we will make the relatively weak non-parametric assumption that the noise vectors $(N_{m,i})_{i\in[n_m]}$ are mutually independent, and the distribution of each noise vector $N_m$ is sign-symmetric, independently of all other noise vectors, i.e., for any vector $b\in \{\pm1\}^{n_m}$, $(N_{m,1}, \ldots, N_{m,n_m}) =_d (b_1N_{m,1}, \ldots, b_{n_m}N_{m,n_m})$.}

{We consider the randomization test from equation \eqref{t0}, where we randomly flip the sign of the datapoints $K$ times using diagonal matrices $B_{m,i}$, $i=1,\ldots,K$, with iid Rademacher entries on the diagonal. We have the following result.
}

{\begin{proposition}[Consistency of randomization test for sparse vector detection]\label{e1}
Let $X_{m,i}=s_{m}+N_{m,i}$, $i=1,\ldots,n_m$, where $s_m$ are $p_m$-dimensional signal vectors and $N_{m,i}$, $i=1,\ldots,n_m$, are mutually independent vectors such that $N_{m,i}=_d-N_{m,i}$. 
As $m\to\infty$, the sequence of randomization tests \eqref{t0} of the sequence of null hypotheses
$s_m=0$, with statistics $f_m(X_m)=n_m^{-1}\|1_{n_m}^\top X_m\|_\infty$ and randomization distribution uniform over $n_m\times n_m$ diagonal matrices with independent Bernoulli entries is consistent 
against the sequence of alternatives with $s_{m} \in \Theta_{m1}$,
if there is a sequence $(t_m)_{m\ge 1}$ such that with probability tending to unity, $\|n_m^{-1}\sum_{i=1}^{n_m} N_{m,i}\|_\infty\le t_m$, and
 for any sequence $(s_m)_{m\ge 1}$ such that for all $m\ge 1$, $s_m \in \Theta_{m1}$,
\beq\label{cg}\liminf_{m\to\infty}\frac{\|s_m\|_\infty}{2t_m}>1.
\eeq
\end{proposition}}

    {See Section \ref{pfe1} for the proof. Roughly speaking, this result shows the consistency of the signflip-based randomization test when the signal strength is at least ``twice above the noise level", as formalized in equation \eqref{cg}. Intriguingly, Proposition \ref{consdet} leads to the same condition; thus suggesting that the additional noise created by randomization is small in this case.}

{{\bf Obtaining consistency results.} Therefore, {obtaining specific consistency results} boils down to controlling $\|n_m^{-1}\sum_{i=1}^{n_m} N_{m,i}\|_\infty$, the $\ell_\infty$ norm of a mean of potentially non-iid random vectors. This can be accomplished under a variety of conditions, and has been widely studied in the areas of concentration inequalities and empirical processes. 
We need to find $t_m$ such that $\|n_m^{-1}\sum_{i=1}^{n_m} N_{m,i}\|_\infty\le t_m$ holds with probability tending to unity. \\
Consider first the simplest setting:  for all $m \ge 1$ and $i\in [m]$, $N_{m,i}$ are iid and have $p_m$ iid coordinates sampled from a probability distribution $\pi$. Then by a union bound, the required condition holds with probability least $1-p_m F_\pi(t_m; n_m)$, where
$F_\pi(t,n) = P(|n^{-1}\sum_{i=1}^{n} Z_{i}|> t)$, with $Z_i\sim \pi$ iid for all $i\in [n]$. 
To ensure consistency, it is thus enough if $t_m$ is such that $\lim_{m\to\infty} p_m F_\pi(t_m; n_m) = 0$.
The tail bounds from Section \ref{tail} imply the following:}
{
 \benum
   \item For \emph{sub-exponential random variables} (including  \emph{sub-gaussian and bounded variables}), Bernstein-type inequalities  imply
    $\lim_{m\to\infty}$ $p_m$ $F_\pi(t_m; n_m) = 0$ if $t_m \sim \sqrt{(\log p_m)/n_m}$, assuming $t_m\le 1$. 
   \item For random variables with a finite $\Psi$-Orlicz norm and finite variance, where $\Psi:\R^+ \mapsto \R^+$, $\Psi(x) = \exp(\ln[x+1]^\kappa)-1$, the results of \cite{chamakh2021orlicz} imply that
   $\lim_{m\to\infty}$ $p_m$ $F_\pi(t_m; n_m) = 0$ if $t_m \sim \exp[(\log p_m)^{1/\kappa}]/\sqrt{n_m}$. 
   \eenum   }

 {{\bf Non-iid noise vectors with possibly dependent entries.} Beyond the simplest setting of iid noise vectors with iid entries, one can consider more general, non-identically distributed noise vectors with possibly dependent entries. The  sign-symmetry requirement $N_m:=(N_{m,1}, \ldots, N_{m,n_m}) =_d (b_1N_{m,1}, \ldots, b_{n_m}N_{m,n_m})$ for the validitity of the randomization test is equivalent to taking an arbitrary random vector $N_m^0=(N_{m,1}^0, \ldots, N_{m,n_m}^0)$, and then multiplying each $N_{m,j}^0$, $j\in [n_m]$, by an independent Rademacher random variable.\\
 To bound the tail of such a test statistic $f_m(N_m)$ for arbitrary noise distribution, one general approach is to first condition on the ``orbit'' of $N_m$ under the signflip group, $G(N_m)  = \{(v_1N_{m,1}, \ldots, v_{n_m}N_{m,n_m}), v\in \{\pm1\}^{n_m}\}$, apply a bound accounting for the random signflips (possibly using bounds on Bernoulli processes), and finally control the resulting tail bound over the unconditional distribution of $N_m$. 
 }


{\bf Rate-optimality.} 
{
Next, using tools from Section \ref{ro}, we discuss certain rate-optimality results for the randomization tests discussed in this section.
In the setting of Proposition \ref {e1}, consider $P_{m0}$ specifying the distribution of the noise $N_m$, and $P_{mi} \in H_{m1}$, $i=1,\ldots, M_m$. 
Then 
  $$L_m = \frac{1}{M_m} \sum_{j = 1}^{M_m} \frac{p_{mj}(X_{m})}{p_{m0}(X_{m})}
   = \frac{1}{M_m} \sum_{j = 1}^{M_m} \frac{p_{m0}(X_m-1_{n_m}S_{mj}^\top)}{p_{m0}(X_{m})}
   = \frac{1}{M_m} \sum_{j = 1}^{M_m} \prod_{i=1}^{n_m} \frac{p_{m,i,0}(X_{m,i}-S_{mj})}{p_{m,i,0}(X_{m,i})}.$$
Suppose all $N_{m,i}$ have equal distribution, with $p_m$ iid coordinates with density $\pi$. Let $M_m = p_m$, and  $S_{mj} = \tau_m \cdot e_j$, where $e_j$ is the $j$-th standard basis vector, and $\tau_m>0$ will be chosen below. Then
  $$L_m
   = \frac{1}{p_m} \sum_{j = 1}^{p_m} \prod_{i=1}^{n_m} \frac{\pi(X_{m,i,j}-\tau_m)}{\pi(X_{m,i,j})}.$$
 Thus,
   $$\V{L_m}
   = \frac{1}{p_m}\V\left[\prod_{i=1}^{n_m} \frac{\pi(X_{m,i,1}-\tau_m)}{\pi(X_{m,i,1})}\right]
   =\frac{1}{p_m}\left\{\left(\V_{Z\sim \pi}\left[\frac{\pi(Z-\tau_m)}{\pi(Z)}\right]+1\right)^{n_m}-1\right\}.$$
 Under appropriate regularity conditions in parametric statistical models  
   $$\V_{Z\sim \pi}\left[\frac{\pi(Z-\tau_m)}{\pi(Z)}\right]
   =\chi^2(\pi(\cdot-\tau_m),\pi) = I_\pi\cdot \tau_m^2 + o(\tau_m^2),$$
   where $I_\pi  = \int \pi'(x)^2/\pi(x) dx$ is the Fisher information of $\pi$ \cite[see e.g.,][Theorem 7.12.]{polyanskiy2019info}. Consistency requires that $\lim_{m\to\infty}\V_{P_{m0}}[L_m]= \infty$, so that for any $C>0$,  $\lim_{m\to\infty} (1+I_\pi\tau_m^2)/\log(Cp_m+1) \ge 1$. Thus, the minimal signal strength required for detection is at least $\sim\sqrt{\log(p_m)/n_m}$. For sub-exponential random variables, this shows that the signflip randomization test is rate-optimal in this case.\\
To summarize this discussion, we can formulate the following result:
\begin{proposition}[Rate-optimality of signflip test for sparse vector detection]\label{e1r}
Under the assumptions of Proposition \ref{e1}, suppose that $N_{m,i}$, $i=1,\ldots,n_m$, have iid entries from a distribution $\pi$ that is sub-exponential and symmetric about zero. 
Let $\Theta_{m1}(\tau_m) = \{s_m \in \R^{p_m}: \|s_m\|_\infty \ge \tau_m\}$. 
The sequence of signflip-based randomization tests \eqref{t0} of the sequence of null hypotheses
$s_m=0$ from Proposition \ref{e1} 
is 
consistent 
against the sequence of alternatives with $s_{m} \in \Theta_{m1}(\tau_m)$
when $\tau_m = C \sqrt{\log(p_m)/n_m}$ for a sufficiently large constant $C>0$.
Moreover, when $\tau_m = o(\sqrt{\log(p_m)/n_m})$, there is no consistent sequence of tests of $s_m=0$ against $s_{m} \in \Theta_{m1}(\tau_m)$.
\end{proposition}
}

\subsubsection{Spherical noise}
\label{sph-vec}

{We also study the case of spherical noise. Since the symmetry group of the noise is larger, it turns out that is enough to have a single observation $X_m = s_m + N_m \in \R^{p_m}$ to obtain a consistent test for a reasonable signal strength.}
We consider the randomization test from equation \eqref{t0}, with a randomization distribution that rotates the data $K$ times using uniformly chosen rotation matrices $O_{m,i}\in O(p_m)$, $i=1,\ldots,K$.

\begin{proposition}[Consistency of orthogonal randomization test for sparse vector detection]\label{e1s}
Let $X_m = s_m+N_m$, where $X_m,s_m,N_m$ are $p_m$-dimensional vectors and {$N_m$ has a spherical distribution}. As $m\to\infty$, the sequence of randomization tests \eqref{t0} with statistics $\|X_m\|_\infty$ and randomization distributions uniform over $O(p_m)$ is consistent {against the sequence of alternatives with $s_{m} \in \Theta_{m1}$, if there is a sequence $(t_{m,2})_{m\ge 1}$ such that with probability tending to unity, $\|N_m\|_2\le t_{m,2}$, and  for any sequence $(s_m)_{m\ge 1}$ such that for all $m\ge 1$, $s_m \in \Theta_{m1}$,
\beq\label{svdr}
\liminf_{m\to\infty}\frac{\|s_m\|_\infty/(2\log p_m)^{1/2}}{\left(\|s_m\|_2 + 2t_{m,2}\right)/p_m^{1/2} }>1.
\eeq
}
\end{proposition}

See Section \ref{pfe1s} for the proof.

\begin{figure}[htb!]
\centering
\includegraphics[width=0.49\textwidth]{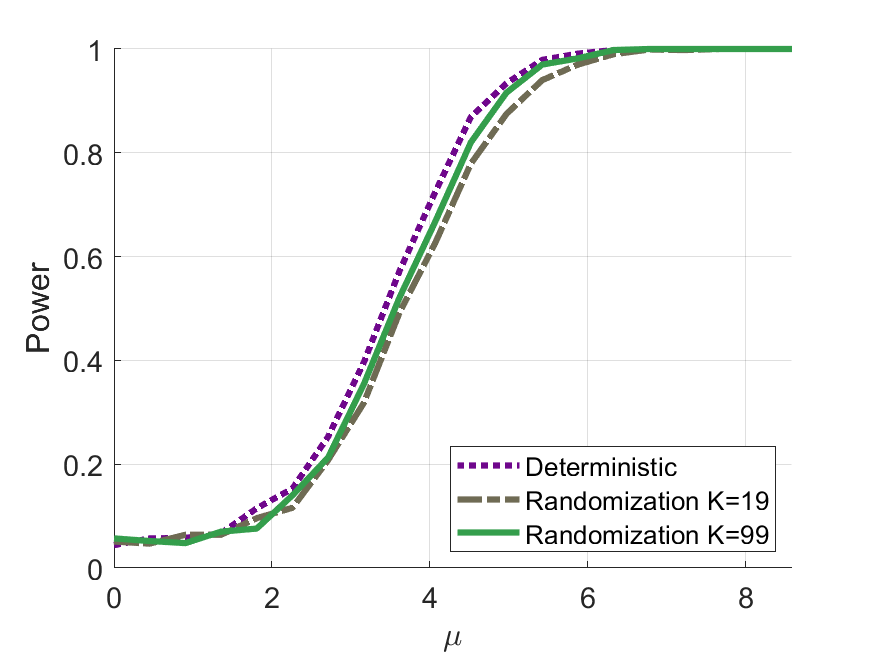}
\includegraphics[width=0.49\textwidth]{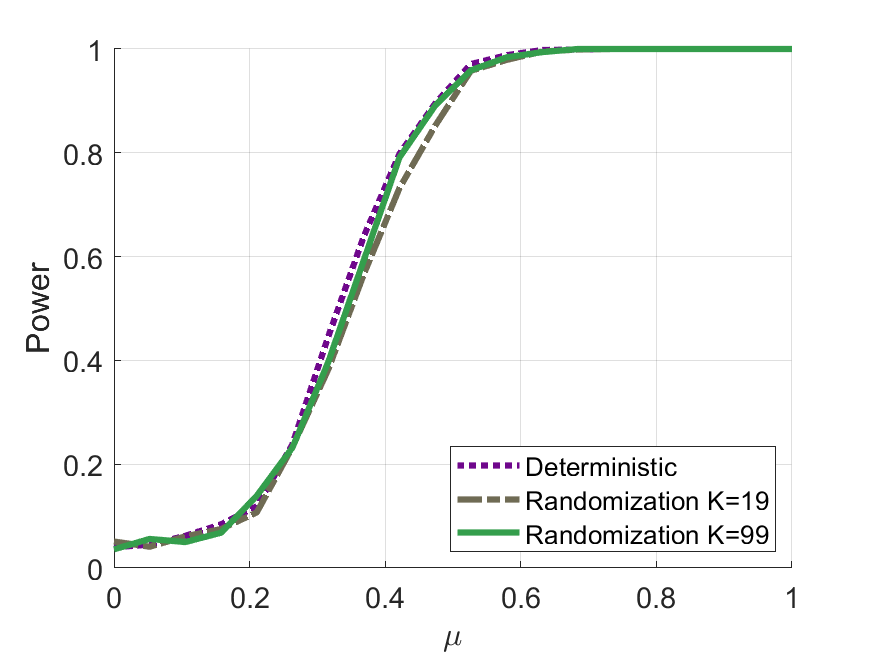}
\caption{
Evaluating the power of the randomization test in comparison with the deterministic test as a function of signal strength in sparse vector detection. Left plot: rotation test; right plot: signflip test. See the text for details.}
\label{f_sp_or}
\end{figure}

This condition is a form of \emph{relative sparsity}: the maximal absolute coordinate $\|s_m\|_\infty$ is large compared to the $\ell_2$ norm  $\|s_m\|_2$ and to the noise level $\|N_m\|_2$. 
{Proposition \ref{consdet} leads to the condition $\liminf_{m\to\infty}\|s_m\|_\infty/\|N_m\|_\infty \ge 2$. Now, one can check (and we do in the proof) that $\|N_m\|_\infty  =_d \|N_m\|_2 \cdot \|Z_m\|_\infty/\|Z_m\|_2$, where $Z_m \sim \N(0,I_{p_m})$. Moreover, as we also check in the proof, $ \|Z_m\|_\infty \sim (2\log p_m)^{1/2}$ and $ \|Z_m\|_2 \sim p_m^{1/2}$. Hence, the condition for the deterministic test is (roughly)
$$\liminf_{m\to\infty}\frac{\|s_m\|_\infty/(2\log p_m)^{1/2}}{2t_{m,2}/p_m^{1/2} }>1.$$
We can see that the condition is milder that \eqref{svdr} (compare the denominators); but may be asymptotically equivalent if $\|s_m\|_2 = o(t_{m,2})$.
}

\begin{table}[]
\renewcommand\arraystretch{1.2}
\centering
\begin{tabular}{|l|l|l|}
\hline
Distribution                           & Density                            & Distribution of $\|Z\|^2$ \\ \hline
Normal                                 & $\sim \exp(-\|z\|_2^2/2)$        & $\chi_p^2$                       \\ \hline
Multivar. Cauchy                    & $\sim (1+\|z\|_2^2)^{-(p+1)/2}$ & $p\cdot F_{p,1}$                       \\ \hline
Multivar. $t$ with $d$ d.o.f.       & $\sim (1+\|z\|_2^2/d)^{-(p+d)/2}$ & $p\cdot F_{p,d}$                       \\ \hline
\end{tabular}
\caption{Classical examples of spherical distributions, for random vectors $Z\in \R^p$, for $p>0$. The densities are given up to constants independent of the argument $z\in \R^p$, and the distribution of $\|Z\|_2^2$ is given in terms of classical distributions such as the chi-squared distribution with $p$ degrees of freedom ($\chi_p^2$), and the $F$-distribution with $p$ and $d>0$ degrees of freedom ($F_{p,d}$).}
\label{sph:t}
\end{table}

{{\bf Obtaining consistency results.} Therefore, {obtaining specific consistency results} boils down to controlling $\|N_{m}\|_2$, the $\ell_2$ norm of a spherically invariant random vector. 
This distribution can be completely arbitrary.
We give a few examples of such random vectors in Table \ref{sph:t}, including normal, multivariate $t$, and multivariate Cauchy distributions. See \cite{fang2018symmetric}, Chapter 3, for more examples.}
{
   \benum
   \item For $Z_m\sim \N(0,I_{p_m})$, we have $\|Z_m\|_2^2 \sim \chi^2_{p_m}$. 
   By the chi-squared tail bound in Lemma~8.1 of \citet{Birge2001}, when $\Gamma_m \sim \chi^2_{p_m}$
   $$\mathbb{P}\biggl(\Gamma_m/p_m \geq 1 + 2\sqrt{\frac{x}{p_m}} + \frac{2x}{p_m}\biggr) \le e^{-x}.$$
   Hence, for any sequence $(l_m)_{m\ge1}$ such that $l_m>0$ for all $m$ and $l_m \to \infty$ as $m\to \infty$,  $\Gamma_m ^{1/2} \le p_m^{1/2} \left(l_m^{1/2} \wedge [1+ O((l_m/p_m)^{1/2})]\right)$ with probability tending to unity. 
   Thus we can take $t_{m,2} = p_m^{1/2} \left(l_m^{1/2} \wedge [1+ O((l_m/p_m)^{1/2})]\right)$. 
   \item For a multivariate Cauchy distribution (and more generally a multivariate $t$ distribution with $d_m\ge 1$ degrees of freedom), by the chi-squared tail bound in Lemma~8.1 of \citet{Birge2001}, when $\Gamma_m \sim \chi^2_{d_m}$, $\Gamma_m/d_m \ge 1 - 2\sqrt{x/d_m}$ with probability at most $\exp(-x)$. Hence, for any sequence $(l_m)_{m\ge1}$ such that $l_m>0$ for all $m$ and $l_m \to \infty$ as $m\to \infty$,  $1/\Gamma_m ^{1/2} \le d_m^{1/2} \left(l_m^{1/2} \wedge [1+ O((l_m/d_m)^{1/4})]\right)$ with probability tending to unity. 
   Thus we can take 
   $$t_{m,2} = \frac{p_m^{1/2} \left(l_m^{1/2} \wedge [1+ O((l_m/p_m)^{1/2})]\right)}{d_m^{1/2} \left(l_m^{1/2} \wedge [1+ O((l_m/d_m)^{1/4})]\right)}.$$
   \eenum
   }

{{\bf Discussion of rate-optimality.}
In this case, obtaining explicit lower bounds on detection thresholds is much more difficult.
We are not aware of any results in this direction under the full level of generality of our model, and thus we discuss the difficulties here.
Suppose that the noise distribution has density $p_m$ with respect to the Lebesgue measure; since the distribution is rotationally invariant, we have $p_m(N_m) = \pi_m(\left\|N_m\right\|_2)$ for some density $\pi_m$ on $[0,\infty)$. The chi-squared method shows that to achieve consistency, one must have 
$$\lim_{m\to\infty} \int_{x_m \in \R^{p_m}} \frac{ \pi_m(\left\|x_m-s_m\right\|_2)^2}{ \pi_m(\left\|x_m\right\|_2)} dx_m =\infty.$$
For instance, if the noise is distributed as a multivariate $t$ distribution with $d_m$ degrees of freedom, with density $c_{m}(1+\|z\|_2^2)^{-(p_m+d_m)/2}$, where $c_m = \Gamma[(p_m+d_m)/2]/[\Gamma(d_m/2) (\pi d_m)^{p_m/2}]$, then we must show that, with $e_m = (p_m+d_m)/2$,
$$\lim_{m\to\infty} 
c_{m}\int_{x_m \in \R^{p_m}} 
\left(\frac
{1+\|x_m\|_2^2/d_m} 
{(1+\|x_m-s_m\|_2^2/d_m)^{2}}
\right)^{e_m}
dx_m =\infty.$$
By changing variables to $x_m-s_m$, using the rotational invariance of the density, denoting $\nu_m = \|s_m\|$, we can express the integral as an expectation with respect to $X_m$ distributed as a multivariate $t$ distribution with $d_m$ degrees of freedom as
$$ 
\E
\left(1 +\frac
{\nu_m(2X_{m,p_m}+\nu_m)} 
{d_m+\|X_m\|_2^2}
\right)^{e_m}.$$
However, there does not appear to be a simple way to evaluate, or obtain sharp bounds on, this expectation, showing the difficulty of obtaining lower bounds for this problem.}

{\bf Numerical example.} We support our theoretical result by a numerical example. We generate data from the signal-plus-noise	model $X_m= s_m+N_m$, where $N_m\sim \N(0,I_{p_m})$, with $p_m=100$ and $s_m = (\mu,0,0,\ldots,0)^\top$ with the signal strength parameter $\mu$ taking values over a grid of size $20$ spaced equally between 0 and $4\cdot \sqrt{\log p_m}$. 
We evaluate the power of the deterministic test based on $\|X_m\|_\infty$, tuned to have level equal to $\alpha=0.05$. 
The critical value $t_\alpha$ is set so that $P_{H_{m0}}(\|X_m\|_\infty\ge t_\alpha)=0.05$, and thus equals $t_\alpha = \Phi^{-1}([(1-\alpha)^{1/p_m}+1]/2)$, where $\Phi^{-1}$ is the standard normal quantile function, i.e., the inverse of the standard normal cumulative distribution function. {In this case, the noise has both rotational and sign symmetry.}
We also evaluate the power of the randomization test based on $K=19$ {and $K=99$} random orthogonal rotations {as well as the same number of random signflips}, with $\alpha=0.05$. We repeat the experiment {1000} times and plot the average frequency of rejections.

On Figure \ref{f_sp_or}, we observe that, as expected, the randomization tests correctly controls the level (under the null when $\mu=0$). Moreover, the power of all tests increases to unity over the range of signals considered, and the deterministic test has only slightly higher power than the randomization tests. In particular, the randomization tests achieve power almost equal to unity at almost the same point as the deterministic test. This is aligned with our results, and supports our claims that the randomization tests are near-optimal. {Further, we also observe that the power with $K=99$ random transforms is slightly higher.}

\begin{figure}[htb!]
\centering
\includegraphics[width=0.49\textwidth]{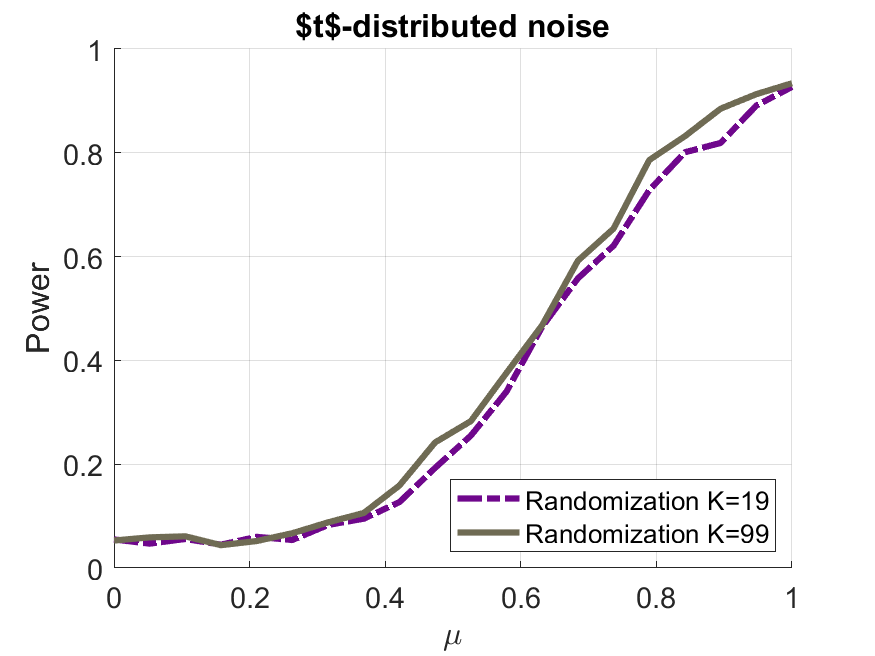}
\includegraphics[width=0.49\textwidth]
{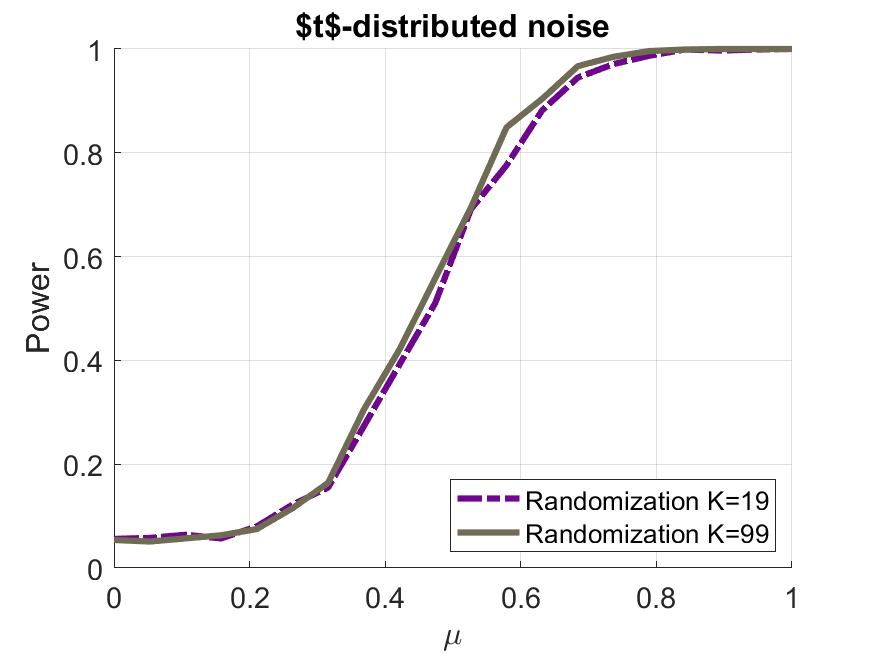}
\caption{
Evaluating the power of the randomization test for $t$-distributed noise. Left plot: $t$ distribution with three degrees of freedom; right plot:  $t$ distribution with five degrees of freedom. See the text for details.}
\label{f_t}
\end{figure}

{{\bf Heavy tailed example.} One of the the strengths of randomization tests is that they seamlessly apply to heavy tailed noise. To illustrate this, we repeat the above experiment with $t$-distributed noise entries (with three and five degrees of freedom, respectively) instead of normal noise, and using the signflip randomization test. 
On Figure \ref{f_t}, we observe that the power of the randomization test increases over the range studied; but since the $t$ distribution has heavier tails than the normal, the power increases at a slower rate than in our previous experiment, especially for the $t$ distribution with three degrees of freedom.}




\subsection{Detecting spikes/low-rank matrices}
A second example is the important problem of detecting low-rank matrices, which is fundamental in multivariate statistical analysis, including in PCA and factor analysis, see e.g., \cite{anderson1958introduction,muirhead2009aspects,johnstone2001distribution,dobriban2017permutation,johnstone2015testing,johnstone2018pca,hong2020selecting}.

Here the data $X_m$ is represented as an $n_m \times p_m$ matrix, where often $n_m$ is the number of samples/datapoints, and $p_m$ is the number of features. We are interested to detect if there is a latent signal in the highly noisy observation matrix; and we model this by a matrix with a large operator norm. 
Formally, $X_m=s_m+N_m$, where $s_m,N_m$ are $n_m\times p_m$ matrices, and we use the  operator norm test statistic $f_m(X_m) = \|X_m\|_\op = \sigma_{\max}(X_m)$. {This is just one of the many possibilities. One could consider other $\psi$-subadditive test statistics; and in particular norms, such as the maximum absolute entry, $\max_{i,j}|X_{m,ij}|$, or generalized Ky Fan norms of the form $X \mapsto (\sum_{i=1}^\kappa \sigma_i(X_m)^\zeta)^{1/\zeta}$, where $\sigma_1(X_m) \ge \ldots \sigma_{n_m \wedge p_m}(X_m)\ge 0$ are the singular values of $X_m$, $\kappa \ge 1$, and $\zeta \ge 1$ \citep{li1988some}.} 

{As in the previous sections, there are many possible models for the structure of the noise and its corresponding group of invariances. For illustration, we only study one of them here. We consider a model where the columns of $N_m$ are independent, and each has a spherical distribution.}
As in the general theory, we consider a sequence of such signal-plus-noise matrices, for a sequence of signals $s_m$. We can then randomize via independent uniform rotations of the columns.
Recall that $\|s_m\|_{2,\infty}$ is the maximum of the column $\ell_2$ norms of $s_m$.


\begin{proposition}\label{e2}
Let the observations follow the matrix signal-plus-noise model $X_m = s_m+N_m$, where $X_m,s_m,N_m$ are $n_m \times p_m$-dimensional matrices and 
{each column of $N_m$ is independent, with a spherical distribution}.
As $n_m,p_m\to \infty$ such that $c_0\le n_m/p_m\le c_1$ for arbitrary fixed $0<c_0<c_1$, the sequence of randomization tests \eqref{t0} with test statistics $\|X_m\|_\op$ and  randomization distributions  uniform over the direct product of orthogonal groups $\mathcal{G}_m = O(n_m) \otimes O(n_m) \ldots \otimes O(n_m)$ rotating the columns of the data is consistent  {against the sequence of alternatives with $s_{m} \in \Theta_{m1}$, if there is a sequence $(t_{m,2})_{m\ge 1}$ such that with probability tending to unity, $\|N_m\|_{2,\infty}\le t_{m,2}$, and for any sequence $(s_m)_{m\ge 1}$ such that for all $m\ge 1$, $s_m \in \Theta_{m1}$,}
{$$\liminf_{m\to\infty}\frac{\|s_m\|_\op/(n_m^{1/2}+p_m^{1/2})}
{(\|s_m\|_{2,\infty}+2t_{m,2})/n_m^{1/2}}>2.$$}
\end{proposition}
See Section \ref{pfe2} for the proof.
{One can verify that Proposition \ref{consdet} implies that the deterministic test based on $\|\hbeta_m\|_\infty$ is consistent when
$$\liminf_{m\to\infty}\frac{\|s_m\|_\op/(n_m^{1/2}+p_m^{1/2})}
{2t_{m,2}/n_m^{1/2}}>2.$$}

{When $N_m\sim\N(0,I_{n_m}\otimes I_{p_m})$, one can verify that we can take $t_{m,2} = n_m^{1/2} (1 + o_P(1))$, thus the condition in Proposition \ref{e2} can be verified to simplify to $\liminf_{m\to\infty}[\|s_m\|_\op/(n_m^{1/2}+p_m^{1/2})-
\|s_m\|_{2,\infty}/n_m^{1/2}]>1.$}

{More generally, suppose that $N_m = [\nu_{m,1} O_{m,1}; \ldots, \nu_{m,p_m} O_{m,p_m}]$, where $\nu_{i,m_i}$, $i\in [p_m]$ are iid from a distribution with cdf $F_m$, and $O_{i,m_i}$, $i\in [p_m]$ are iid according to the Haar measure on the orthogonal group $O(p_m)$. Then the condition on $t_{m,2}$ is that $P(\max_{i=1}^{p_m} \nu_{i,m_i} \le t_{m,2}) = F_m(t_{m,2})^{p_m} \to 1$. Consider any sequence $(l_m)_{m\ge1}$ such that $l_m>0$ for all $m$ and $l_m \to 0$ as $m\to \infty$. Then, we can take $t_{m,2} =  F_m^{-1}(1- l_m/p_m)$. }

{{\bf Rate-optimality.} Suppose that $N_m\sim\N(0,I_{n_m}\otimes I_{p_m})$, and let $\Theta_{m1} = \{\sqrt{n_m/2} \cdot\tau\cdot uv^\top, v \in \R^{n_m}, u\in \R^{p_m}, \|u\|=\|v\|=1 \}$. 
Suppose without loss of generality that $n_m \le p_m$; otherwise flip the roles of $n_m$ and $p_m$.
Consider a prior $\Pi_m$ on $\Theta_{m1}$ such that $u = [v; 0_{p_m-n_m}]$, and $v$ follows a distribution $\Pi_m'$. Based on \eqref{vlr}, we have
\begin{align*}
\V_{P_{m0}}[L_m]  &= \E_{S,S'\sim \Pi_m} \exp(S^\top S')
= \E_{uv^\top,u'(v')^\top\sim \Pi_m} \exp(n_m\tau^2/2\cdot u^\top u' v^\top v')\\
&= \E_{v,v'\sim \Pi_m'} \exp(n_m\tau^2/2\cdot (v^\top v')^2)
.
\end{align*}
This has the exact same form as the expression studied in Theorem 1 of \cite{banks2018information}. From that result, it follows that, if $\Pi_m'$ is uniform over $\{\pm1\}^{n_m}/\sqrt{n_m}$ and $\tau<1$, then  $\V_{P_{m0}}[L_m] \le C$ for a constant $C<\infty$ not depending on $n_m$. This shows a lower bound of order $\tau \gtrsim n_m^{1/2}$. Meanwhile, our upper bound simplifies to $\tau \lesssim n_m^{1/2}$, showing that randomization tests are rate-optimal in this case.\\
To summarize:
\begin{proposition}[Rate-optimality of rotation test for low-rank matrix detection]\label{e2r}
Under the assumptions of Proposition \ref{e2}, suppose that 
$N_m\sim\N(0,I_{n_m}\otimes I_{p_m})$, and let $\Theta_{m1}(\tau_m) = \{s_m=\sqrt{\min(n_m,p_m)/2} \cdot\tau_m\cdot uv^\top, v \in \R^{n_m}, u\in \R^{p_m}, \|u\|=\|v\|=1 \}$.
The sequence of rotation tests 
\eqref{t0}
of the sequence of null hypotheses
$s_m=0$ from Proposition \ref{e2} 
is 
consistent 
against the sequence of alternatives with $s_{m} \in \Theta_{m1}(\tau_m)$
when $\tau_m = C \sqrt{\min(n_m,p_m)}$ for a sufficiently large constant $C>0$.
Moreover, when $\tau_m = o(\sqrt{\min(n_m,p_m)})$, there is no consistent sequence of tests of $s_m=0$ against $s_m\in\Theta_{m1}(\tau_m)$.
\end{proposition}}


\subsection{Sparse detection in linear regression}

We consider the fundamental linear regression problem $Y_m = X_m\beta_m+\ep_m$, where $\ep_m$ is random. The null hypothesis is that $\beta_m=0$, and we are interested to detect ``sparse" alternatives in the same way as in Section \ref{se1}, i.e., vectors $\beta_m$ with a large $\ell_\infty$ norm.

We can directly view this as a signal plus noise model, where $s_m = X_m\beta_m$.
 However, the most direct approach of using a test statistic such as $f_m(Y_m) = \|Y_m\|_\infty$ leads to a condition for consistency that depends on the $\ell_\infty$ norm $X_m\beta_m$ as opposed to $\beta_m$ only. 
Instead, we write the ordinary least squares (OLS) estimator $\hbeta_m$ as 
\begin{align*}
\hbeta_m = X_m^\dagger Y_m = P_{X_m}\beta_m+X_m^\dagger\ep_m,
\end{align*}
where $X_m^\dagger$ is the pseudo-inverse of $X_m$, and $P_{X_m}$ is the projection into the row space of $X_m$. 
Formally, this is the OLS estimator if $n_m \ge p_m$ and $X_m$ has full rank; otherwise it is the minimum $\ell_2$ norm interpolator of the normal equations $X_m^\top (Y_m-X_m \hbeta_m) = 0$.
We can view this as a signal-plus-noise model with observation $X'_m = \hbeta_m$, signal $s_m = P_{X_m}\beta_m$, and noise $N_m = X_m^\dagger\ep_m$.
If $n_m \ge p_m$ and $X_m$ has full rank, $s_m = \beta_m$, but in general this approach only provides information about the projection of $\beta_m$ into the row span of $X_m$.
We are interested to detect sparse signals using the test statistic $f_m(\hbeta_m)  = \|\hbeta_m\|_\infty$. 

{As before, there are many possibilities for the structure of the noise. As in Section \ref{sign-noise}, we consider coordinate-wise sign-symmetric noise, assuming that for any vector $b\in \{\pm1\}^{n_m}$, $(\ep_{m,1}, \ldots, \ep_{m,n_m}) =_d (b_1\ep_{m,1}, \ldots, b_{n_m}\ep_{m,n_m})$.
We consider the randomization test from equation \eqref{t0}, where we randomly flip the sign of the data $Y_m$ $K$ times using diagonal matrices $B_{m,i}$, $i=1,\ldots,K$, with iid Rademacher entries on the diagonal.
For any $n_m$-dimensional vector $v$, define  the matrix
\beq\label{xm}
\mathcal X_m(v) =  [X_m^\dagger\diag(v);-X_m^\dagger\diag(v)].
\eeq
For $j=1,\ldots,p_m$, let $[X_m^\dagger]_{j,\cdot}$ be the $j$-th row of $X_m^\dagger$. 
Let 
\beq\label{tx}
T(X_m) = \{\diag([X_m^\dagger]_{j,\cdot})X_m w : w\in \R^{p_m}, \|w\|_\infty\le 1,  j\in [p_m]  \}.
\eeq
Define the vector $|\ep_m| = (|\ep_{m,1}|,\ldots,|\ep_{m,n_m}|)^\top$. Recall $U^+$ from \eqref{u+}.
}
Below, $\|M\|_{\infty,\infty} = \sup_{\|v\|_\infty \le 1} \|Mv\|_\infty$ is the induced matrix norm, which is also the maximum of the $\ell_1$ norms of the rows of $M$.

\begin{proposition}\label{e3} Let the data $(X_m,Y_m)$ follow the linear regression model $Y_m = X_m\beta_m+\ep_m$, where $Y_m$ is an $n_m$-dimensional vector of outcomes, $X_m$ is and $n_m \times p_m$-dimensional observation matrix, and $\beta_m$ is an unknown $p_m$-dimensional vector of regression parameters. 
{Let $\ep_m$ have independent entries $\ep_{m,i}$, $i=1,\ldots,n_m$, such that $\ep_{m,i}=_d-\ep_{m,i}$.}
The sequence of randomization tests 
\eqref{t0}
of the null hypothesis $P_{X_m}\beta_m=0$ with test statistics $\|\hbeta_m\|_\infty$, where $\hbeta_m = X_m^\dagger Y_m$, and randomization distributions uniform over {$n_m\times n_m$ diagonal matrices with independent Bernoulli entries} is consistent
{against the sequence of alternatives with $P_{X_m}\beta_{m} \in \Theta_{m1}$, 
if there are two sequences $(l_m)_{m\ge1}$ and  $(t_{m})_{m\ge1}$ such that the following hold:
\benum
\item 
 $l_m>0$ for all $m$ and $l_m \to \infty$ as $m\to \infty$,
\item  $t_m>0$ for all $m$ and, with $U^+$ from \eqref{u+} and $\mathcal X_m$ from \eqref{xm}, $P(U^+(\mathcal X_m(|\ep_m|), l_m ) \le t_{m})\to 1$,
\item 
  for any sequence $(P_{X_m}\beta_m)_{m\ge 1}$ such that for all $m\ge 1$, $P_{X_m}\beta_m \in \Theta_{m1}$, with $T(X_m)$ from \eqref{tx},
$$\liminf_{m\to\infty}\left(\|P_{X_m}\beta_m\|_\infty \frac{1-U^+(T(X_m), l_m)}{2t_m}\right)> 1.$$
\eenum}
\end{proposition}
See Section \ref{pfe3} for the proof. 
{This result bounds the quantity $\|X_m^\dagger \ep_m\|_\infty$ by an ``asymmetrization'' argument first, by conditioning on $|\ep_m|$ and using the Bernoulli/Rademacher randomness over the signs of the entries of $\ep_m$. However, in specific cases when more is known about the distribution of $\ep_m$, one may obtain simpler results by directly bounding this quantity. For instance, when $\ep_m \sim \N(0, I_{p_m})$, $X_m^\dagger \ep_m \sim \N(0, X_m^\dagger (X_m^\dagger)^\top)$, and under certain structural conditions on $X_m$, one may be able to derive sharp bounds for the required maximum $\|X_m^\dagger \ep_m\|_\infty$ of a correlated multivariate Gaussian random vector.
}

{For comparison, one can verify that Proposition \ref{consdet} implies that the deterministic test based on $\|\hbeta_m\|_\infty$ is consistent when the (at least as liberal) condition $\liminf_{m\to\infty}\|P_{X_m}\beta_m\|_\infty/(2t_m)> 1$ holds. }\\

{{\bf Discussion of rate-optimality.} There is a large literature on optimal hypothesis testing for linear regression, see for instance \cite{ingster2010detection,arias2011global,mukherjee2020minimax,carpentier2021optimal} and references therein. These works essentially only study iid Gaussian (or sub-Gaussian) noise, and make varying assumptions on the design matrix and signal strength. In general it appears quite difficult to make a direct comparison to our assumptions. For instance the work of \cite{arias2011global} (their Theorem 2) implies that if $[X_m]_{j,\cdot}$ is the $j$-th row of $X_m$, and $(c_m)_{m\ge1}$ is a sequence  such that $c_m>0$ for all $m$ and $c_m \to 0$ as $m\to \infty$, then if $X_m^\top X_m$  is normalized to have unit diagonal entries,  if for all $i\in [p_m]$, $|\{j\in [p_m]: |[X_m]_{j,\cdot}^\top [X_m]_{i,\cdot}| \ge c_m (\log p_m)^{-4}\}| = O(p_m^\delta)$ for all $\delta >0$, and if the regression coefficient $\beta_m$ can be any 1-sparse vector, then  it is required that $\liminf_{m\to\infty}\|\beta_m\|_\infty/\sqrt{2\log p_m}\ge 1$ in order for any test to have non-vanishing detection power. The main assumption is that for any feature, the number of other features with correlation above the level $c_m (\log p_m)^{-4}$ is smaller than any positive power of $p_m$. This assumption does not appear to be easily comparable to our conditions. Indeed, our conditions require (among others) to bound $\|X_m^\dagger \ep_m\|_\infty$, where $X_m^\dagger \ep_m \sim \N(0, X_m^\dagger (X_m^\dagger)^\top)$, which does not appear to be directly related to the conditions from \cite{arias2011global}. \\
Thus, our conditions under which the randomization test works appear to be different from the ones that have been studied before for rate optimality in this problem. Since our main goal in this paper was to develop a general framework that enables proving consistency results for randomization tests, we view it as beyond our scope to fully elucidate the relationships between our conditions and those variously proposed in the literature. We would like to emphasize that our consistency results cover settings where the noise for every observation is assumed to be merely independent and symmetrically distributed, potentially heteroskedastic and heavy-tailed. This goes beyond the settings in which lower bounds have been proved for this problem.}

\subsection{Two-sample testing}

We study a two-sample testing problem, which is a classical and fundamental problem of exceeding importance in statistics,  see e.g., \citep{lehmann1998theory,lehmann2005testing}. We study this for illustration purposes only, as there are well-established tests. We do not claim that randomization tests are better, merely that they are applicable, and it is of interest to understand what they lead to.

We consider permutation based randomization tests, valid when the entries of the noise are exchangeable. {For a given integer $m \ge 1$ and dimension $p_m$, let $(f_\mu)_{\mu \in \R^{p_m}}$ be a location family of densities on $\R^{p_m}$. Let $\|\cdot\|_{\R^{p_m}}$ be a norm on $\R^{p_m}$. Let 
$\ep_{m,i} \sim f_{0_m}$ sampled from the location family at the all-zero vector be iid for $i\in [n_m]$, and $\ep_{m,i}' \sim f_{0_m}$ also be iid for $i\in [n_m']$. } 

\begin{proposition}\label{e5}
{Suppose $Z_{m,1},\ldots,Z_{m,n_m}\sim f_{\mu_m}$, $Y_{m,1},\ldots,Y_{m,n_m'}\sim f_{\mu_m'}$ are independent observations, and test the null hypothesis that $\mu_m=\mu_{m}'$ against the alternative that $\mu_{m}\neq \mu_m'$.} 
{Consider the randomization test \eqref{t0} with test statistic $\|\bar Z_{m}-\bar Y_{m'}\|_{\R^{p_m}}$, where $\bar Z_{m} = n_m^{-1}\sum_{i=1}^{n_m} Z_{m,i}$ and $\bar Y_{m}' = ({n_m'})^{-1}\sum_{i=1}^{n_m'} Y_{m,i}$.}

For a randomization distribution uniform over the symmetric group of all permutations $S_{n_m+n_m'}$, the sequence of randomization tests \eqref{t0} of the sequence of null hypotheses $\mu_m=\mu_{m}'$ is consistent 
{against the sequence of alternatives with $(\mu_m,\mu_{m}') \in \Theta_{m1}$, if
\benum
\item 
 as $m\to\infty$, $n_m+n_m'\to\infty$,
\item 
 there is a sequence $(t_{m})_{m\ge1}$ such that $P(\|({n_m'})^{-1}\sum_{i=1}^{n_m'} \ep_{m,i}' - n_m^{-1}\sum_{i=1}^{n_m} \ep_{m,i}\|_{\R^{p_m}}\le t_{m})\to 1$, 
where  $\ep_{m,i}, \ep_{m,j}'  \sim f_{0_m}$, $i\in [n_m]$, $j\in [n_m']$ are iid.
\item 
for any sequence $(\mu_m,\mu_{m}')_{m\ge 1}$ such that for all $m\ge 1$, $(\mu_m,\mu_{m}') \in \Theta_{m1}$,
$$\liminf_{m\to\infty}
\frac{\|\mu_{m}'-\mu_m\|_{\R^{p_m}}}{t_m}>2.$$
\eenum}
\end{proposition}
See Section \ref{pfe5} for the proof. 
{As for the one-sample test for sparse detection, Proposition \ref{consdet} leads to the same condition; thus suggesting that the additional noise due to randomization is small.}
{The condition looks similar to the one we obtained for the one-sample test; however this concerns a different randomization distribution (permutations), and thus requires a different analysis. 
Bounding $t_m$ depends on the conditions we impose on the location family, on the growth of the dimension and sample sizes, and on the specific norm used. For instance, in certain cases one may use Orlicz-norm based concentration inequalities (see e.g., Section \ref{sign-noise} for examples), which can be adapted to the norm $\|\cdot\|_{\R^{p_m}}$. \\
Following the approach from Section \ref{sign-noise}, for $\|\cdot\|_{\R^{p_m}} = \|\cdot\|_\infty$, the same results stated there apply by assuming the same conditions on the noise vectors for both samples, and by bounding the noise vectors of the two samples separately. For instance, if the entries of $\ep_{m,i}$, $i\in [n_m]$, $\ep_{m,i}'$, $i\in [n_m']$ are iid sub-exponential, then we can take $t_m \sim \sqrt{(\log p_m)/\min(n_m, n_m')}$.}

{{\bf Rate-optimality.} It is straightforward to see that the lower bound technique from Section \ref{sign-noise} generalizes, and leads to a bound of the order $\tau_m=\sqrt{(\log p_m)/\min(n_m, n_m')}$. Indeed, when $n_m \le n_m'$, one can take $\mu_{m',j}=0$ and $\mu_{m,j} = \tau_m \cdot e_j$, for $j\in [n_m]$ in the construction of the alternatives in Ingster's method, and it is straightforward to see that the desired conclusion holds by the same calculation as in Section \ref{sign-noise}.  This shows that for noise with iid sub-exponential entries, the signflip based randomization test is rate-optimal. To summarize:
\begin{proposition}[Rate-optimality of permutation test for sparse two-sample testing]\label{e5r}
Under the assumptions of Proposition \ref{e5}, suppose that 
$\ep_{m,i} \sim f_{0_m}$ for $i\in [n_m]$, and $\ep_{m,i}' \sim f_{0_m}$ for $i\in [n_m']$, have iid entries with a sub-exponential  distribution $\pi$. 
Let $\Theta_{m1}(\tau_m) = \{(\mu_m,\mu_{m}') \in \R^{p_m}\times\R^{p_m}: \|\mu_m-\mu_{m}'\|_\infty \ge \tau_m\}$. 
The permutation test of the sequence of null hypotheses
$\mu_m=\mu_{m}'$ from Proposition \ref{e5} 
is 
consistent 
against the sequence of alternatives with $(\mu_m,\mu_{m}') \in \Theta_{m1}(\tau_m)$
when $\tau_m = C \sqrt{\log(p_m)/\min(n_m, n_m')}$ for a sufficiently large constant $C>0$.
Moreover, when $\tau_m = o(\sqrt{\log(p_m)/\min(n_m, n_m')})$, there is no consistent sequence of tests of $\mu_m=\mu_{m}'$ against $(\mu_m,\mu_{m}') \in \Theta_{m1}(\tau_m)$, $m\ge 1$.
\end{proposition}
}

\begin{figure}[htb!]
\centering
\includegraphics[width=0.5\textwidth]{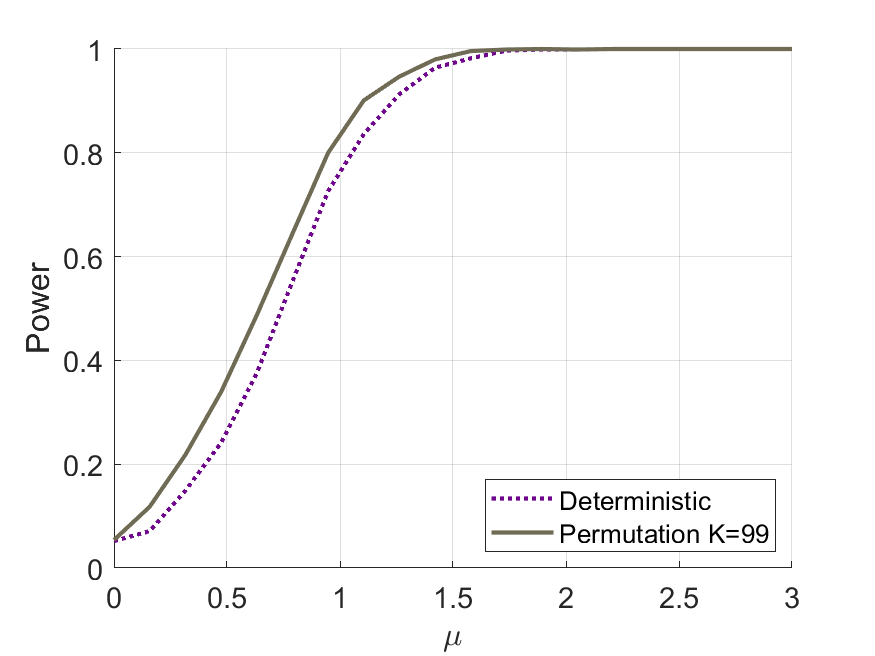}
\caption{
Evaluating the power of a permutation test in comparison with the $t$-test as a function of signal strength in two-sample testing. See the text for details.}
\label{f_t_perm}
\end{figure}

{{\bf Numerical example.} We support our theoretical result by a numerical example, using the two-sample $t$-test.\footnote{We thank a referee for suggesting this experiment.} We generate data from the Gaussian signal-plus-noise	model $Z_{m,i} \sim \N(s_m,1)$, for $i\in [n_m]$, and $Y_{m,i} \sim \N(0,1)$, for $i\in [n_m']$, where $s_m = \mu$, with the signal strength parameter $\mu$ taking values over a grid of size $20$ spaced equally between 0 and 3. 
We take $n_m = n_m'=15$.
We evaluate the power of the deterministic test based on the two-sample $t$-test, tuned to have level equal to $\alpha=0.05$. 
We also evaluate the power of the randomization test based on $K=99$ random permutations. We repeat the experiment 1000 times and plot the average frequency of rejections.}

{On Figure \ref{f_t_perm}, we observe similar phenomena to those mentioned before: the randomization test correctly controls the level, and the power of both tests increases to unity over the range of signals considered. The power of the two tests is very close.\footnote{We note that similar observations have been made by \cite{lehmann2012parametric}.}  In this experiment, the permutation test even has a slightly higher power.}




\section{Discussion}
\label{disc}
{
	We developed a set of results on the consistency of randomization tests. While we think that our results are quite powerful, they also have a number of limitations to be addressed in future work:
	\benum
	\item A limitation is the restriction to signal plus noise models. This is needed in the current proof technique; in fact our entire approach is based on this structure. However, to broaden the scope of our results, it would be important to extend to more general statistical models.
	\item Another limitation is that the level $\alpha$ is considered fixed. This is also needed in the proof, and is specifically used in the bound \eqref{fixk}. In some applications, especially in multiple hypothesis testing, the level $\alpha$ needs to shrink with the problem size. It would be important to extend our theory to this setting.
	\eenum
}

\section*{Acknowledgments}

We thank Edward I. George, Jesse Hemerik, Panos Toulis, and Larry Wasserman for valuable discussions. This work was supported in part by NSF BIGDATA grant IIS 1837992 and NSF CAREER award DMS 2046874.

\section{Appendix}
\label{pf}

\subsection{Practical Considerations}
\label{prac}

When are invariance based tests applicable in practice? When can one invoke the group invariance hypothesis?
We think that this is a challenging applied statistics problem, and we provide some discussion here.\footnote{We thank a reviewer for raising this question.}
When a data analyst is performing a hypothesis test, and they have reason to think that under the null hypothesis the distribution of the data is (nearly) unchanged under some operation, then one can invoke a group invariance condition.
Suppose for instance that the data analyst thinks that under the null hypothesis, the data is equally likely to have come in any order --- then one can invoke permutation invariance.
However, suppose that the data comes in predefined clusters (such as strata, or classes based on some key distinguishing class), and under the null hypothesis it is only reasonable to think that that data is equally likely to appear in any order in some specific clusters. Then one can use permutation invariance only over the permutations within those clusters.

This type of reasoning is more readily justifiable when testing a point null. In that case, since we only consider one distribution, assumptions can be justified with greater ease. However if we consider composite null hypotheses, such as those in two-sample testing, then it becomes much more challenging to justify invariance assumptions.

However one difficulty is that formally testing (evaluating) invariance assumptions can be very difficult, especially if the invariance groups are large (for instance suppose that we  only have one observation; then it is impossible to test that its density is symmetric around zero).
In our view these type of decisions can be quite application-specific.
Further there are a number of books and reviews on group invariance and permutation tests in statistics, and the interested statistical data analyst can study them for additional insights  \citep[see e.g.,][etc.]{pesarin2001multivariate,ernst2004permutation,pesarin2010permutation,pesarin2012review,good2006permutation,kennedy1995randomization,eaton1989group,wijsman1990invariant,giri1996group}.
\color{black}

\subsection{Proof for the general theory}

\subsubsection{Proof of $\psi$-sub-additivity in Section \ref{gencons}}
\label{pf-psi}

{
Let $x,y\in \R$ and suppose first that $0\le x < y$. Then, by concavity, $c(y) = c(x[x/y]+[x+y][1-x/y]) \ge c(x)(x/y)+c(x+y)(1-x/y)$, or equivalently, $c(x+y) \le [y c(y) - xc(x)]/[y-x]$. Thus, $c(x+y)\le c(x)+c(y)$ follows if $xc(y)\le yc(x)$. By concavity again, and also using that $c(0)\ge 0$, we have $c(x) = c(y[x/y]+0[1-x/y]) \ge c(y)(x/y)+c(0)(1-x/y) \ge c(y)(x/y)$, as required. Next, if $0\le x = y$, then the above argument used for $y = 2x$ shows that $c(x) \ge c(2x)/2$, thus $c(x+y) = c(2x)\le 2c(x) = c(x)+ c(y)$. This finishes the argument when $x,y\ge 0$. The same argument applies when $x,y\le 0$.\\
The remaining case is when $x,y$ have opposite signs. We can assume without loss of generality that $x<0<y$ and that $|y| \ge |x|$ (otherwise we can consider $(-x,-y)$). Then $f(x+y) = c(|x+y|) = c(x+y)  = c(y-|x|) \le c(y) + c(|x|)$, where the last inequality follows because $c$ is non-decreasing, and also as $0\le c(0)\le c(|x|)$.}

\subsubsection{Proof of Theorem \ref{cons}}
\label{pfcons}

{\bf Control of type I error.}
The first claim, about the level/Type I error control, is discussed at various levels of generality in many works. The textbook result, e.g., Problem 15.3 in \cite{lehmann2005testing} considers finite groups, and for infinite groups (e.g., Problem 15.1 in the same reference), assumes that we average over the full group. See also more general statements in theorem 2 in \cite{hemerik2018false} and theorem 2 in \cite{hemerik2018exact}.   We provide a simple argument {to show a key required exchangeability claim,} which extends the above results allowing for compact topological groups at a full level of generality, and applies to random sampling of a finite number of group elements. This is crucial for our results, because we use continuous groups such as orthogonal groups in many of our examples.

Let $T_0 = f_m(N_m)$, and $T_i = f_m(G_{mi}N_m)$ for $i=1,\ldots,K$. Note that due to noise invariance, $T_i$, $i=0,\ldots,K$ are exchangeable when $N_m,G_{m1},\ldots, G_{mK}$ are all considered random: the random variables in the vector $L=(N_m,G_{m1}N_m,\ldots,G_{mK}N_m)$ are exchangeable.

\begin{lemma} The random vectors $\{N_m,G_{m1}N_m,\ldots$, $G_{mK}N_m\}$ are mutually exchangeable.
\end{lemma}

\begin{proof}
 To see this, we will show that $L=(N_m,G_{m1}N_m,\ldots,G_{mK}N_m)$ has the same distribution as $B=(G_mN_m,G_{m1}N_m,\ldots,G_{mK}N_m)$, where $G_m\sim Q_m$ is independent of $N_m,G_{m1},\ldots$, $G_{mK}$. Denote $G_mN_m = N_m'$. Then this is equivalent to the statement that $A_m$ has the same distribution as $(N_m',G_{m1}G_m^{-1}N_m',\ldots,G_{mK}G_m^{-1}N_m')$. 

Let $G_{mi}' = G_{mi}G_m^{-1}$, for $i=1,\ldots,K$. Since $N_m= _d N_m'$, the above claim follows from because the vectors $(G_{m1},\ldots,G_{mK})$ and $(G_{m1}',\ldots,G_{mK}')$ have an identical distribution. 
For simplicity, we show this for $K=2$. The proof for the more general case is very similar. 

We can write for $i\neq j$, $Q_m(G_{mi}' \in M_i, G_{mj}' \in M_j)=Q_m(G_{mi} G_m^{-1} \in M_i, G_{mj}G_m^{-1} \in M_j) = Q_m(G_{mi} \in G_mM_i, G_{mj} \in G_mM_j)$. Now, let us condition on $G_m$. Then, we can write using the independence of $G_{mi},G_{mj}$ that  $Q_m(G_{mi} \in G_mM_i, G_{mj} \in G_mM_j|G_m) = Q_m(G_{mi} \in G_mM_i|G_m) Q_m(G_{mj} \in G_mM_j|G_m)$. Recall that $G_{mi} \sim Q_m$ are iid from the Haar/uniform probability measure on $\mathcal{G}_m$. Using the left-invariance of the Haar measure, we have $ Q_m(G_{mi} \in G_mM_i|G_m) = Q_m(G_{mi} \in M_i|G_m) = Q_m(G_{mi} \in M_i)$, and similarly for $j$. Hence, we find, using again the independence of $G_{mi}, G_{mj}$ that
$$Q_m(G_{mi}' \in M_i, G_{mj}' \in M_j) =  Q_m(G_{mi} \in M_i) Q_m(G_{mj} \in M_j) 
= Q_m(G_{mi} \in M_i, G_{mj} \in M_j).$$
This shows that the joint distribution of $(G_{m1},\ldots,G_{mK})$ and $(G_{m1}',\ldots,G_{mK}')$ is the same for $K=2$. The same argument works for $K>2$. This finishes the proof.
\end{proof}
{One can then finish the proof of type I error control as in the proof of theorem 2 in \cite{hemerik2018exact}.}

\vspace{5mm} 
{\bf Consistency.}
Now we move to the part about consistency.
We will consider a slight variant of the invariance-based randomization test, where for a fixed $K\ge 1$ we reject the null when 
\beq\label{t}
f_m(X_m)>\max\left(f_m(G_{m1}X_m),\ldots, f_m(G_{mK}X_m)\right),
\eeq
and where each $G_{mi}$, $i=1,\ldots,K$ is chosen uniformly at random over $\mathcal{G}_m$. The type I error probability over the random $X_m$ and $G_{mi}$ of this test is at most $1/(K+1)$, see Theorem \ref{cons}. 
{The consistency of this test implies the consistency of the quantile-based test. Specifically, given any $\alpha\in(0,1)$, choose any positive integer $K$ such that $1/(K+1)\le \alpha$. Let $R_{m,K}$ denote the event \eqref{t} and let $R_{m,\alpha}$ denote the event \eqref{t0}. Then, $R_{m,K} \subset R_{m,\alpha}$, and hence $P_{H_{m1}}(R_{m,K}) \le P_{H_{m1}}(R_{m,\alpha})$. We will show that $P_{H_{m1}}(R_{m,K}) \to 1$. Thus, it will follow that $P_{H_{m1}}(R_{m,\alpha}) \to 1$. Therefore, it is enough to study the test \eqref{t}.}
A simplification is given by the following lemma.

\begin{lemma}\label{r1}
Suppose $K$ is fixed. Then we have $P(f_m(X_m) > \max_{i=1}^K f_m(G_{mi}X_m)) \to 1$ if and only if we have $P(f_m(X_m) >  f_m(G_mX_m)) \to 1$ for a single $G_m\sim Q_m$.
\end{lemma}

\begin{proof}[Proof of Lemma \ref{r1}]
Consider the events $A_i=\{f_m(X_m) \le f_m(G_{mi}X_m)\}$. By taking complements, it is enough to show that $P(\cup_{i=1}^K A_i) \to 0$ if and only if $P(A_1)\to 0$.

Since $G_{mi}$ have the same distribution for all $i\in[k]$, we have  $P(A_i) = P(A_j)$ for all $i,j$. Moreover, since $A_1\subset \cup_{i=1}^K A_i$, we have by the union bound that
\beq\label{fixk}
P(A_1) \le P(\cup_{i=1}^K A_i) \le \sum_{i=1}^K P(A_i)  = K \cdot P(A_1).
\eeq
Hence, as $K$ is bounded, we have $P(\cup_{i=1}^K A_i) \to 0$ iff $ P(A_1)\to 0$.
\end{proof}

Thus, for consistency to hold, it is enough to show that with probability tending to unity,
\begin{align*}
f_m(X_m) > f_m(G_mX_m).
\end{align*}
Now, $f_m(G_mX_m)=f_m(G_ms_m+G_mN_m)$.  We have the following:

\begin{lemma}[Independence Lemma]\label{il}
If $g_mN_m= _d N_m$ for any fixed $g_m\in \mathcal{G}_m$, then $G_m\indep G_mN_m$ when $G_m\sim Q_m$.
\end{lemma}

\begin{proof}[Proof of Lemma \ref{il}]
We can write, for a measurable set $A$
\begin{align*}
P(G_mN_m\in A|G_m=g_{m0})&=P(g_{m0}N_m\in A|G_m=g_{m0})
=P(g_{m0}N_m\in A)
=P(N_m\in A).
\end{align*}
Since this expression does not depend on $g_{m0}$, the distribution of $G_mN_m$ does not depend on the value of $G_m$; thus $G_mN_m$  is independent of $G_m$.
\end{proof}

This implies that  for $G_m,N_m$  sampled  independently, $G_ms_m+G_mN_m$ has the same distribution as $G_ms_m +N_m$. Therefore,   $f_m(G_mX_m)=_df_m(G_ms_m+N_m)$, and it is enough to give conditions for
the potentially stronger condition that there is a deterministic sequence of critical values $t_m'$ such that 
\begin{align}\label{thresh}
P_{H_{m1}}(f_m(G_ms_m+N_m)\le t_m') + P_{H_{m1}}(f_m(X_m)> t_m')\to 2.
\end{align}

{By $\psi$-subadditivity, we can write
\begin{align}\label{psi-s}
f_m(X_m) = f_m(s_m+N_m) &\ge \psi f_m(s_m) -f_m(-N_m).
\end{align}
Since $t_m$  is such that $P(f_m(-N_m) \le t_m)\to 1$, we conclude that $P(f_m(X_m) \ge \psi f_m(s_m) - t_m)\to 1$. Hence, if
$f_m(s_m)> \psi^{-1}[t_m'(s_m) + t_m],$
then the desired condition $P_{H_{m1}}(f_m(X_m)> t_m')\to 1$ holds, provided that $P_{H_{m1}}(f_m(G_ms_m+N_m)\le t_m')\to 1$. 
By $\psi$-subadditivity again, we can write
\begin{align*}
f_m(G_ms_m+N_m) &\le \psi^{-1}[f_m(G_ms_m) + f_m(N_m)]
\le \psi^{-1}[\tilde t_m + t_m]
.
\end{align*}
Taking $t_m' = \psi^{-1}[\tilde t_m + t_m]$
finishes the proof.}

\subsubsection{Proof of Proposition \ref{consdet}}
\label{pfcons2}
{
As in the proof of Theorem \ref{cons},
it is enough to give conditions for the analogue of \eqref{thresh}, i.e., 
that there is a deterministic sequence of critical values $t_m'$ such that 
\begin{align*}
P_{H_{m0}}(f_m(N_m)\le t_m') + P_{H_{m1}}(f_m(X_m)> t_m')\to 2.
\end{align*}
By condition 2(a) of Theorem \ref{cons}, we can take $t_m'=t_m$, and $P_{H_{m0}}(f_m(N_m)\le t_m') \to 1$. By $\psi$-subadditivity, we have \eqref{psi-s}. Thus, we only need that $\psi f_m(s_m) - t_m > t_m$, which is true by \eqref{cd}. This shows that we can take $\tilde c_m \le t_m$ and finishes the proof.
}

\subsection{Proofs for the examples}
\subsubsection{Proof of Proposition \ref{e1}}
\label{pfe1}

{Since $\|\cdot\|_\infty$ is a norm, it is 1-subadditive.
Thus, the condition from Theorem \ref{cons} reads $n_m^{-1}\|1_{n_m}^\top s_m\|_\infty> \tilde t_m(s_m) + 2t_m.$ 
Moreover, $n_m^{-1}\|1_{n_m}^\top s_m\|_\infty = \|s_m\|_\infty$. 
The requirement on $t_m, \tilde t_m$ is that with probability tending to unity, $\|n_m^{-1}\sum_{i=1}^{n_m} N_{m,i}\|_\infty\le t_m$, 
and for Rademacher random variables $b_{m,i}$, $i\in[n_m]$, with probability tending to unity, $\|n_m^{-1}\sum_{i=1}^{n_m} b_{m,i} s_{m}\|_\infty = |n_m^{-1}\sum_{i=1}^{n_m} b_{m,i} |\cdot\|s_{m}\|_\infty\le \tilde t_m$. \\
By Hoeffding's inequality, for any $C>0$, $P(|n_m^{-1}\sum_{i=1}^{n_m} b_{m,i}| \ge C) \le 2\exp(-2n_m C^2)$. Hence, we can take $\tilde t_m = (a_m/[2n_m])^{1/2} \cdot \|s_m\|_\infty$, for any sequence $(a_m)_{m\ge 1}$ with $a_m\to\infty$. 
Thus, the condition is that for all $m$ large enough,
$$\|s_m\|_\infty> (a_m/[2n_m])^{1/2} \cdot \|s_m\|_\infty + 2t_m.$$
This requires that $a_m/[2n_m]<1$, which we can ensure holds for all large enough $n_m$ by taking $a_m$ to grow sufficiently slowly. For such large $n_m$, the condition is
$$\|s_m\|_\infty> \frac{2t_m}{1-(a_m/[2n_m])^{1/2}}.$$
Clearly, this holds when $a_m$ grows sufficiently slowly, for instance when $a_m  = \log n_m$, if $\lim\inf_{m\to\infty}\frac{\|s_m\|_\infty}{2t_m}>1$.
}

\subsubsection{Proof of Proposition \ref{e1s}}
\label{pfe1s}

{Since $\|\cdot\|_\infty$ is a norm, it is 1-subadditive.
Thus, the condition from Theorem \ref{cons} reads $\|s_m\|_\infty> \tilde t_m(s_m) + 2t_m.$ 
The requirement on $t_m,\tilde t_m$ is that with probability tending to unity, $\|N_{m}\|_\infty\le t_m$, 
and for $O_m\sim O(p_m)$, with probability tending to unity, $\|O_m s_{m}\|_\infty \le \tilde t_m$. \\
Now, for a normal random vector $Z_m \sim \N(0,I_{p_m})$, we have $\|O_m s_{m}\|_\infty =_d \|Z_m\|_\infty/\|Z_m\|_2 \cdot \|s_m\|_2$.
 For $Z_m\sim \N(0,I_{p_m})$, using standard chi-squared concentration of measure \citep{boucheron2013concentration}, we have $\|Z_m\|_2 = p_m^{1/2}(1+o_P(1))$.
Moreover, $\|Z_m\|_\infty \le (1+o_P(1))\sqrt{2\log p_m}$ with probability tending to unity.
Hence, we can take $\tilde t_m = (1+o_P(1))(2[\log p_m]/p_m)^{1/2} \cdot \|s_m\|_2$. 
Similarly, $\|N_m\|_\infty = \|O_m N_{m}\|_\infty =_d \|Z_m\|_\infty/\|Z_m\|_2 \cdot \|N_m\|_2 = (1+o_P(1))(2[\log p_m]/p_m)^{1/2} \cdot \|N_m\|_2$.\\
Thus, the condition is that there is a sequence $t_{m,2}$ such that $P(\|N_m\|_2\le t_{m,2})\to 1$ and for all $m$ large enough,
$$\|s_m\|_\infty> (1+o_P(1))(2[\log p_m]/p_m)^{1/2} \cdot \left(\|s_m\|_2 + 2t_{m,2}\right).$$
This holds when
$$\liminf_{m\to\infty}\frac{\|s_m\|_\infty/(2\log p_m)^{1/2}}{\left(\|s_m\|_2 + 2t_{m,2}\right)/p_m^{1/2} }>1.$$
}

\subsubsection{Proof of Proposition \ref{e2}}
\label{pfe2}

Since the maximal singular value is a norm, it is 1-subadditive.
Thus, the condition from Theorem \ref{cons} reads $\|s_m\|_\op> \tilde t_m(s_m) + 2t_m.$ 
The requirement on $t_m,\tilde t_m$ is that with probability tending to unity, $\|N_{m}\|_\op\le t_m$, 
and for $O_{m,1},\ldots, O_{m,p_m}\sim O(n_m)$, with probability tending to unity, $\|[O_{m,1}s_{m,1};\ldots; O_{m,p_m}s_{m,p_m}]\|_\op \le \tilde t_m$.

Now, for iid normal random vectors $Z_{m,i} \sim \N(0,I_{n_m})$, $i\in[p_m]$, we have $O_{m,i} s_{m,i} =_d  Z_{m,i}/\|Z_{m,i}\|_2 \cdot \|s_{m,i}\|_2$.
Thus,
$$\|[O_{m,1}s_{m,1};\ldots; O_{m,p_m}s_{m,p_m}]\|_\op
=_d
\|[Z_{m,1}/\|Z_{m,1}\|_2 \cdot \|s_{m,1}\|_2;\ldots; Z_{m,p_m}/\|Z_{m,p_m}\|_2 \cdot \|s_{m,p_m}\|_2]\|_\op
.$$
Further, for any matrix $M = [m_1; m_2; \ldots; m_{p_m}]$ and scalars $d_i$, $i\in[p_m]$,
\begin{align*}
\|[d_1 m_1; d_2 m_2; \ldots; d_m m_{p_m}]\|_{\op} \le \max_{i}{|d_i|}\cdot \|M\|_\op.
\end{align*}

Now, from standard concentration inequalities we have $P(|\|Z_{m,i}\|/n_m^{1/2} - 1|\ge \delta + 1/\sqrt{n_m}) \le 2\exp(-n_m\delta^2/2)$. This follows from the Lipschitz concentration of Gaussian random variables, see e.g., Example 2.28 in \cite{wainwright2019hds}, and from the fact that the mean of the $\chi(n_m)$ random variable $\|Z_{m,i}\|$ is bounded as $\sqrt{n_m}-1\le \E \|Z_{m,i}\|\le \sqrt{n_m}$, see exercise 3.1 in \cite{boucheron2013concentration}. 

Taking a union bound, we find that $P(\max_{i=1,\ldots,p_m}|\|Z_{m,i}\|_2/n_m^{1/2} - 1|\ge \delta +  1/\sqrt{n_m}) \le 2\exp(\log p_m-n_m\delta^2/2)$. So, $\max_{i=1,\ldots,p_m}|\|Z_{m,i}\|_2/n_m^{1/2} - 1| \to_P 0$ as long as there is a sequence $\delta=\delta_{m}$ such that $\delta_{m}\to 0$ and $n_m\delta_{m}^2-2\log p_m\to \infty$. This holds if $\log p_m = o(n_m)$. Then, we also have that $\max_{i=1,\ldots,p_m}|n_m^{1/2}/\|Z_{m,i}\|_2 - 1|\to_P 0$.

Thus denoting $Z_m =[Z_{m,1};\ldots; Z_{m,p_m}]$, with probability tending to unity,
$$
\|[Z_{m,1}/\|Z_{m,1}\|_2 \cdot \|s_{m,1}\|_2;\ldots; Z_{m,p_m}/\|Z_{m,p_m}\|_2 \cdot \|s_{m,p_m}\|_2]\|_\op
\le 
(1+o_P(1))\|s_m\|_{2,\infty}/n_m^{1/2}\cdot \|Z_m\|_\op
.$$
It is well known that as $n_m,p_m\to \infty$ such that $c_0\le n_m/p_m\le c_1$ for some $0<c_0<c_1$, we have almost surely that $\|Z_m\|_\op \le (1+o_P(1))(\sqrt{n_m}+\sqrt{p_m})$.
This follows from  \cite[][Theorem 2.13]{davidson2001local}.
Hence, we can take $\tilde t_m = (1+o_P(1))\|s_m\|_{2,\infty}(\sqrt{n_m}+\sqrt{p_m})/n_m^{1/2}$. 

Now, due to the distributional invariance of $N_m$, we have 
$$\|N_m\|_\op
=_d
\|[Z_{m,1}/\|Z_{m,1}\|_2 \cdot \|N_{m,1}\|_2;\ldots; Z_{m,p_m}/\|Z_{m,p_m}\|_2 \cdot \|N_{m,p_m}\|_2]\|_\op$$

Hence, using the same argument as above, for any sequence $t_{m,2}$ such that $\|N_m\|_{2,\infty} \le t_{m,2}$ with probability tending to unity, we can take $t_m =  (1+o_P(1))(\sqrt{n_m}+\sqrt{p_m}) \cdot t_{m,2}/n_m^{1/2}$. 
Thus, a sufficient condition is that there is a sequence $t_{m,2}$ such that $P(\|N_m\|_{2,\infty} \le t_{m,2})\to 1$ and
$$\|s_m\|_\op> (1+o_P(1))[1+(p_m/n_m)^{1/2}] \cdot \left(\|s_m\|_{2,\infty}+2t_{m,2}\right).$$
This holds when
$$\liminf_{m\to\infty}\frac{\|s_m\|_\op/(n_m^{1/2}+p_m^{1/2})}
{(\|s_m\|_{2,\infty}+2t_{m,2})/n_m^{1/2}}>1.$$

This finishes the proof.

\color{black}
\subsubsection{Proof of Proposition \ref{e3}}
\label{pfe3}

Since the map $Y_m \mapsto \|X_m^\dagger Y_m\|_\infty$ is a quasi-norm, it is 1-subadditive.
Thus, the condition from Theorem \ref{cons} reads $\|P_{X_m}\beta_m\|_\infty> \tilde t_m + 2t_m.$ 
The requirement on $t_m,\tilde t_m$ is that with probability tending to unity, $\|X_m^\dagger \ep_m\|_\infty\le t_m$, 
and for $B_m = \diag(b_{m,1},\ldots, b_{m,p_m})$ with iid Rademacher entries $b_{m,i}$, $i\in[p_m]$, with probability tending to unity, $\|X_m^\dagger B_m X_m \beta_m\|_\infty \le \tilde t_m$.

Let $(l_m)_{m\ge1}$ be any sequence such that $l_m>0$ for all $m$ and $l_m \to \infty$ as $m\to \infty$.
Now, conditional on the vector $|\ep_m| = (|\ep_{m,1}|,\ldots,|\ep_{m,n_m}|)$, $X_m^\dagger \ep_m$ is an $n_m$-dimensional Bernoulli	 process over the rows of the matrix
$\mathcal X_m(|\ep_m|).$
Thus, conditional on $|\ep_m|$, we have $\|X_m^\dagger \ep_m\|_\infty \le $ $U^+(\mathcal X_m(|\ep_m|), l_m)$ with probability going to unity,  see \eqref{u+}. Thus, it is enough to take $t_m$ to be an upper bound of this quantity with probability tending to unity.

Next, writing $B_m = \diag(b_m)$,
\begin{align*}
\|X_m^\dagger B_m X_m \beta_m\|_\infty 
&\le \|X_m^\dagger B_m X_m \|_{\infty,\infty}\cdot \|\beta_m\|_\infty\\
& = \max_{j\in[p_m]}|[X_m^\dagger]_{j,\cdot}^\top \cdot B_m X_m \|_1
\cdot \|\beta_m\|_\infty\\
& = \max_{j\in[p_m]}\|X_m ^\top \diag([X_m^\dagger]_{j,\cdot}) \cdot b_m  \|_1
\cdot \|\beta_m\|_\infty\\
& = \|\beta_m\|_\infty \cdot  \sup_{v\in T(X_m)} v^\top b_m.
\end{align*}
Thus, it is enough if $\tilde t_m = U^+(T(X_m),l_m)$.
Thus, a sufficient condition is that there is a sequence $(l_m)_{m\ge1}$ such that $l_m>0$ for all $m$ and $l_m \to \infty$ as $m\to \infty$, and a sequence 
 $(t_{m})_{m\ge1}$ such that $P(U^+(\mathcal X_m(|\ep_m|),l_m) \le t_{m})\to 1$ and
$$\liminf_{m\to\infty}\|P_{X_m}\beta_m\|_\infty \frac{1-U^+(T(X_m),l_m)}{2t_m}> 1.$$

This finishes the proof.

\color{black}
\subsubsection{Proof of Proposition \ref{e5}}
\label{pfe5}

We can write $Z_{m,i} = \mu_m + \ep_{m,i}$, for $i\in [n_m]$, where $\ep_{m,i} \sim f_{0_m}$ are iid. 
Similarly, we can write $Y_{m,i} = \mu_m + \ep_{m,i}'$, for $i\in [n_m']$, where $\ep_{m,i}' \sim f_{0_m}$ are also iid.
We can arrange the datapoints as the rows of a matrix. 
This model has a signal-plus-noise form with nuisance $\mu_{m,*} = 1_{n_m+n_m'}\cdot \mu_m^\top $ and signal $S = [0_{n_m}; 1_{n_m'}] \cdot \Delta_m^\top$, where $\Delta_m=\mu_{m}'-\mu_m$.

We can follow our general approach for problems with nuisance parameters, see Section \ref{nuis}.
Let $P_m$ be the projection in the orthogonal complement of the span of the nuisance. We project $X_m = [Z_{m,1};\ldots; Z_{m,n_m}; Y_{m,1};\ldots; Y_{m,n_m'}]$ to  $\tilde X_m = P_mX_m$, and we obtain a standard signal-plus-noise model $\tilde{X}_m  = \tilde{s}_m + \tN_n$.
Since $P_m = I_{n_m+n_m'}-1_{n_m+n_m'}1_{n_m+n_m'}^\top/(n_m+n_m')$, we have $$\tilde{X}_m = [I_{n_m+n_m'}-1_{n_m+n_m'}1_{n_m+n_m'}^\top/(n_m+n_m')] \tilde{X}_m = \tilde{X}_m - 1_{n_m+n_m'}\bar X_m^\top.$$ 
Also 
$$\tilde{s}_m = P_m s_m = s_m - 1_{n_m+n_m'}\bar s_m^\top 
= [- n_m' \cdot 1_{n_m} ; n_m \cdot 1_{n_m'}]/(n_m+n_m') \cdot \Delta_m^\top.$$ 
We can write the test statistic 
$\|\bar Y_m-\bar Z_m\|_{\R^{p_m}}$ as $\|w^\top \tilde{X}_m\|_{\R^{p_m}}$, where $w= [-1_{n_m}/{n_m}; 1_{n_m'}/{n_m'}]$. Note that $P_{m}w=w$.


The test statistic is clearly 1-subadditive.
Thus, the condition from Theorem \ref{cons} reads $\|\Delta_m\|_{\R^{p_m}}> \tilde t_m + 2t_m.$ 
The requirement on $t_m,\tilde t_m$ is that with probability tending to unity, $\|w^\top \tilde N_m\|_{\R^{p_m}}\le t_m$, 
and for a uniformly random  permutation matrix $\Pi_m$ of $n_m+n_m'$ entries, with probability tending to unity, $\|w^\top \Pi_m \tilde s_m\|_{\R^{p_m}} \le \tilde t_m$. 

Now, 
$$\|w^\top \tilde N_m\|_{\R^{p_m}} 
= \|w^\top  N_m\|_{\R^{p_m}}
= \|\bar Y_{m}' -  \bar Z_{m}\|_{\R^{p_m}}
= \|({n_m'})^{-1}\sum_{i=1}^{n_m'} \ep_{m,i}' - n_m^{-1}\sum_{i=1}^{n_m} \ep_{m,i}\|_{\R^{p_m}}.
$$
Also,
$$\|w^\top \Pi_m \tilde s_m\|_{\R^{p_m}} 
= w^\top  \Pi_m w \cdot  \frac{n_m'n_m}{n_m+n_m'}  \|\Delta_m\|_{\R^{p_m}}.
$$
Consider the random variable $U = w^\top  \Pi_m w$, where the randomness is due to the random permutation matrix $\Pi_m$. 
Let $d=n_m+n_m'$ be the dimension of $w$. 
Now, if $\pi_m:[d]\mapsto[d]$ denotes the permutation represented by $\Pi_m$,
\begin{align*}
\E U^2 = \E w^\top  \Pi_m w \cdot w^\top  \Pi_m w
= \E  \sum_{ij} w_iw_{\pi_m(i)}w_jw_{\pi_m(j)}
= \sum_{ij} w_i w_j \E   w_{\pi_m(i)}w_{\pi_m(j)}.
\end{align*}
If $i=j$, then $\E   w_{\pi_m(i)}w_{\pi_m(j)} = \E   w_{\pi_m(i)}^2 = \|w\|^2/d$.
If $i\neq j$, then, since $\sum_k w_k = 0$, 
$$\E   w_{\pi_m(i)}w_{\pi_m(j)} = \frac1{d(d-1)}\sum_{k\neq l}w_{k}w_{l} = - \frac{\|w\|^2}{d(d-1)}.$$ 
Thus,
\begin{align*}
\E U^2
= \sum_{i} w_i^2 \|w\|^2/d
+
\sum_{i\neq j} w_i w_j (- \frac{\|w\|^2}{d(d-1)})
= \|w\|^4 \left(\frac1d+
 \frac{1}{d(d-1)}\right)
 = \frac{\|w\|^4 }{d-1}.
\end{align*}
Now, we can check that $\|w\|^2 =  \frac{n_m+n_m'}{n_m'n_m}$. Therefore, by Chebyshev's inequality, 
$$P(w^\top  \Pi_m w \cdot  \frac{n_m'n_m}{n_m+n_m'} \ge l_m)
=
P(U/\|w\|^2 \ge l_m)
\le \frac{\E (U^2/\|w\|^4)}{l_m^2}
= \frac1{(n_m+n_m'-1) l_m^2}
.
$$
Thus, if $l_m\to\infty$, we can take $\tilde t_m = l_m\cdot \frac{\|\Delta_m\|_{\R^{p_m}}}{(n_m+n_m'-1)^{1/2}}  $.
Thus, a sufficient condition is that there is a sequence $(l_m)_{m\ge1}$ such that $l_m>0$ for all $m$ and $l_m \to \infty$ as $m\to \infty$, and a sequence 
 $(t_{m})_{m\ge1}$ such that $P(\|({n_m'})^{-1}\sum_{i=1}^{n_m'} \ep_{m,i}' - n_m^{-1}\sum_{i=1}^{n_m} \ep_{m,i}\|_{\R^{p_m}}\le t_{m})\to 1$ and
$$\|\Delta_m\|_{\R^{p_m}}> l_m\cdot \frac{\|\Delta_m\|_{\R^{p_m}}}{(n_m+n_m'-1)^{1/2}}  + 2t_m.$$
This requires that $n_m+n_m'\to\infty$. Then, we can take $l_m$ to grow sufficiently slowly, and the above condition holds if 
$$\liminf_{m\to\infty}
\frac{\|\Delta_m\|_{\R^{p_m}}}{t_m}>2.$$

This finishes the proof.
\color{black}

{\small
\setlength{\bibsep}{0.2pt plus 0.3ex}
\bibliographystyle{plainnat-abbrev}
\bibliography{references}
}

\end{document}